%% file: revised_2020_08_07.tex
\documentclass[graybox]{svmult}

\usepackage{mathptmx}       
\usepackage{helvet}         
\usepackage{courier}        
\usepackage{type1cm}        
\usepackage{makeidx}         
\usepackage{graphicx}        
\usepackage{multicol}        
\usepackage[bottom]{footmisc}

\makeindex             

\usepackage{color}
\usepackage{graphicx}
\usepackage{dsfont}
\usepackage{bbm}
\usepackage{amsmath}
\usepackage{amssymb}
\usepackage{amsfonts}
\usepackage{amstext}
\usepackage{enumerate}
\usepackage{amsbsy}
\usepackage{amsopn}
\usepackage{amscd}
\usepackage{amsxtra}
\usepackage{tikz}

\newcommand{\RR}{\mathds{R}}
\newcommand{\NN}{\mathds{N}}

\newcommand{\ind}{\mathds{1}}

\newcommand{\ZZ}{\mathbb{Z}}
\newcommand{\EE}{\mathbb{E}}
\newcommand{\PP}{\mathbb{P}}

\DeclareMathOperator{\Ker}{Ker}
\DeclareMathOperator{\Ran}{Ran}
\DeclareMathOperator{\x}{ \mathbf{x} }

\begin{document}

\title*{Bernoulli Hyperplane Percolation}
\author{Marco Aymone, Marcelo~R.~Hil\'ario, Bernardo~N.~B.~de~Lima and Vladas Sidoravicius}
\institute{Marco Aymone \at Universidade Federal de Minas Gerais, Departamento de Matem\'atica, Av.\ Ant\^onio Carlos 6627, CP 702, 31270-901,  Belo Horizonte, MG, Brazil, \email{aymone@ufmg.br}
\and Marcelo R.\ Hil\'ario \at Universidade Federal de Minas Gerais, Departamento de Matem\'atica, Av.\ Ant\^onio Carlos 6627, CP 702, 31270-901,  Belo Horizonte, MG, Brazil, \email{mhilario@mat.ufmg.br}
\and Bernardo N.B.\ de Lima \at Universidade Federal de Minas Gerais, Departamento de Matem\'atica, Av.\ Ant\^onio Carlos 6627, CP 702, 31270-901,  Belo Horizonte, MG, Brazil, \email{bnblima@mat.ufmg.br}
\and Vladas Sidoravicius \at Shanghai New York University, 1555 Century Avenue, Pudong New District, Shanghai, China 200122.}
%
%
\maketitle

\abstract*{We study a dependent site percolation model on the $n$-dimensional Euclidean lattice where, instead of single sites, entire hyperplanes are removed independently at random.
We extend the results about Bernoulli line percolation showing that the model undergoes a non-trivial phase transition and proving the existence of a transition from exponential to power-law decay within some regions of the subcritical phase.}

\abstract{We study a dependent site percolation model on the $n$-dimensional Euclidean lattice where, instead of single sites, entire hyperplanes are removed independently at random.
We extend the results about Bernoulli line percolation showing that the model undergoes a non-trivial phase transition and proving the existence of a transition from exponential to power-law decay within some regions of the subcritical phase.}

\section{Introduction}

In Bernoulli site percolation on the $\mathbb{Z}^n$-lattice, vertices are removed independently with probability $1-p$.
For $n \geq 2$, the model undergoes a phase transition at $p_c=p_c(\mathbb{Z}^n) \in (0,1)$: For $p<p_c$ all the connected components are finite almost surely whereas, for $p >p_c$, there exists an infinite connected component almost surely  \cite{broadbent57}. 
In a different percolation model on $\mathbb{Z}^n$, $n \geq 3$, called Bernoulli line percolation, instead of single sites, bi-infinite lines (or columns) of sites that are parallel to the coordinate axes are removed independently.
This model, that was introduced in the physics literature by Kantor \cite{Kantor86} and later studied both from the numerical \cite{Grassberger17, Schrenk16} and mathematical \cite{HilSid} points of view, also exhibits a phase transition as the probability of removal of single lines is varied.
However the geometric properties of the resulting connected components differ substantially in these two models.
In fact, while for Bernoulli site percolation the connectivity decay is exponential  except exactly at the critical point \cite{Duminil-Copin16, Aizenman87, Menshikov86}, for Bernoulli line percolation, transitions from exponential to power-law decay occur within the subcritical phase \cite[Theorem 1.2]{HilSid}.
In the present paper we study a higher dimensional version of the Bernoulli line percolation model  that we call {\it Bernoulli hyperplane percolation}. 
In this model, for fixed $n \geq 3$ and $k$ with $1 \leq k \leq n$ we remove from $\mathbb{Z}^n$ entire $(n-k)$-dimensional `affine hyperplanes'.
We introduce the model precisely in the following section.

\subsection{Definition of the model and main results}

In this section we define the Bernoulli hyperplane percolation model.

It can be formulated in terms of orthogonal projections onto ``the coordinate hyperplanes" as follows: For $n \geq 2$ and $1 \leq k \leq n$ we write 
\begin{equation}
\label{e:index_set}
\mathcal{I}=\mathcal{I}(k;n):=\{I  \subset [n]:  \# I = k \},
\end{equation}
where $[n] := \{1, \ldots, n\}$.
For a fixed $I \in \mathcal{I}(k,n)$ we denote $\mathbb{Z}^k_I$ the set of all the linear combinations of the canonical vectors $(e_i)_{i\in I}$ with integer coefficients, that is,
\begin{equation}
\ZZ^k_I:=\Big\{ \sum_{i\in I} x_i e_i \in \mathbb{Z}^n: x_i \in \ZZ \text{ for all $i \in I$} \Big\}.
\label{eq:zkn}
\end{equation}
Since each one of the $\binom{n}{k}$ sets $\mathbb{Z}^k_I$ is isomorphic to the $\mathbb{Z}^k$-lattice they will be called the $k$-dimensional coordinate hyperplanes of $\mathbb{Z}^n$. 
Let us define, independently on each $\mathbb{Z}^k_I$, a Bernoulli site percolation $\omega_I \in \{0,1\}^{\mathbb{Z}^k_I}$ with parameter $p_I\in [0,1]$ that is, a process in which $(\omega_I(u))_{u\in \mathbb{Z}^k_I}$ are independent Bernoulli random variables with mean $p_I \in [0,1]$.
We interpret this as, each site $u \in \mathbb{Z}^k_I$ is removed (that is $\omega_I(u) = 0$) independently with probability $1-p_I$.

Let $\pi_I\colon \mathbb{Z}^n \to \mathbb{Z}^{k}_I$ stand for the orthogonal projection from $\mathbb{Z}^n$ onto $\mathbb{Z}^k_I$
\begin{equation}
\pi_I\bigg{(}\sum_{i=1}^n x_i e_i \bigg{)}:=\sum_{i\in I} x_{i}e_{i}.
\end{equation}
The Bernoulli $(n,k)$-hyperplane percolation on $\ZZ^n$ is the process $\omega = (\omega(v))_{v\in \mathbb{Z}^n} \in \{0,1\}^{\mathbb{Z}^n}$, where
\begin{equation}
\label{eq:def_omega}
\omega(v)=\prod_{I\in \mathcal{I}}\omega_I(\pi_I(v)).
\end{equation}
We denote $\mathbf{p} = (p_I)_{I\in \mathcal{I}}$.
Each entry $p_I \in [0,1]$ is called a parameter of $\mathbf{p}$.
We write $\mathbb{P}_{\mathbf{p}}$ for the law in $\{0,1\}^{\mathbb{Z}^n}$ of the random element $\omega$ defined in \eqref{eq:def_omega}.

Since $\text{$\omega (v) = 1$ if and only if $\omega_I(\pi_I(v)) =1$ for all $I\in \mathcal{I}(k,n)$}$, we may interpret the process $\omega$ in terms of removal of sites in $\mathbb{Z}^n$ as follows: $v$ is removed (that is $\omega(v) =0$) if and only if, for at least one of the $I \in \mathcal{I}(k;n)$ its orthogonal projection into $\mathbb{Z}^k_I$ has been removed in $\omega_I$.
One can easily check that this is equivalent to perform independent removal (or drilling) of $(n-k)$-dimensional hyperplanes that are parallel to the coordinate hyperplanes.

Let $[o \leftrightarrow \infty]$ denote the event that the origin belongs to an infinite connected component of sites $v$ such that $\omega(v)=1$ and $[o \nleftrightarrow \infty]$ its complementary event.
Also denote $[o \leftrightarrow \partial B(K)]$ the event that the origin is connected to some vertex lying at $l_\infty$-distance $K$ from it via a path of sites $v$ such that $\omega(v)=1$.

Our first result generalizes Theorem 1.1 in \cite{HilSid}.
\begin{theorem}
\label{t:phase_transition}
Let $n \geq 3$ and $2\leq k\leq n-1$.
The Bernoulli $(n,k)$-hyperplane percolation model undergoes a non-trivial phase transition, that is:
If all the parameters of $\mathbf{p}$ are sufficiently close to $1$ then $\mathbb{P}_{\mathbf{p}}(o \leftrightarrow \infty)>0$.
On the other hand, when all the parameters of $\mathbf{p}$ are sufficiently close to $0$ then $\mathbb{P}_{\mathbf{p}}(o \leftrightarrow \infty) =0$.
\end{theorem}

The proof of Theorem \ref{t:phase_transition} will be divided into two parts.
The second assertion which concerns the regime in which all the parameters are small is proved in Section \ref{sec:subcritical_phase} (see Remark \ref{remark:proof_phase_tran_2} therein).
In Section \ref{sec:supercritical_phase} we prove the first assertion which concerns the regime in which all the parameters are large (see Remark \ref{remark:proof_phase_tran_1} therein).

Our next result states that for some range of the parameter vector $\mathbf{p}$ the connectivity cannot decay faster than a power law.
\begin{theorem}
\label{theorem power law projetando em k} 
Let $n \geq 3$ and $2\leq k\leq n-1$. 
If, for all $I\in\mathcal{I}(k;n)$, the parameters $p_I<1$ are sufficiently close to $1$, then there exists $c=c(\mathbf{p})>0$ and $\alpha=\alpha(\mathbf{p})>0$ such that
\begin{equation}
\label{e:power-law}
\PP_{\mathbf{p}}(o \leftrightarrow \partial{B}(K),~ o \nleftrightarrow \infty)\geq {c}{K^{-\alpha}}
\end{equation}
for every integer $K>0$.
\end{theorem}

In the special case $k=2$ we can determine more precisely some regions of the parameter space for which power-law decay holds:

\begin{theorem}
\label{t:poly_dec_1} 
Let $n\geq 3$ and $k=2$ and assume that $p_I >0$ for every $I \in \mathcal{I}(k;n)$.
Denote $I_j:=\{1,j\}$ and assume that $p_{I_j}>p_c(\ZZ^2)$  for every $2\leq j\leq n$.
Assume further that $p_I<1$ for some $I\in\mathcal{I}\setminus\{I_2, \ldots, I_n\}$. 
Then there exist $c=c(\mathbf{p})>0$ and $\alpha=\alpha(\mathbf{p})>0$ such that  \eqref{e:power-law} holds for every integer $K>0$.
\end{theorem}

Having stated our main results, we now provide some remarks about the contribution of this paper.

The case $k=1$ does not admit a (non-trivial) phase transition.
In fact, as soon as $p_I >0$ for every $I$, drilling $(n-1)$-dimensional hyperplanes has the effect of splitting the lattice into finite rectangles.
The case $k=n$ corresponds to Bernoulli site percolation (here we interpret $0$-dimensional hyperplanes as being just single sites).
For these reasons in the above statements we have $2\leq k \leq n-1$.
Moreover, the case $k=n-1$ corresponds to the Bernoulli line percolation model studied in \cite{HilSid} so the results are not novel in this specific case.

Theorem \ref{t:poly_dec_1}, that only concerns the case $k=2$, states that when the Bernoulli site percolation processes  $\omega_I$ defined on the $n-1$ coordinate planes that contain $e_1$ are all supercritical, then power-law decay holds regardless of the values of the parameters fixed for the Bernoulli percolation processes in the remaining $\binom{n-1}{2}$ hyperplanes (provided that at least one of these is smaller than $1$). 
It generalizes the first statement in \cite[Theorem 1.2]{HilSid}.
Of course, the arbitrary choice of fixing the `direction' $1$ is made purely for convenience; any other choice would result in an analogous result. 

Still for $k=2$, the same argument used to prove Equation (1.3) in \cite{HilSid} can be employed  to show that if $p_{I} < p_c(\mathbb{Z}^2)$ for at least $\binom{n-1}{2}+1$ parameters, then $\PP_{\mathbf{p}}(o \leftrightarrow \partial{B}(K))$ is exponentially small in $K$, see Remark \ref{remark:exp_dec} for a sketch of the argument.
Hence, like in Bernoulli line percolation, there is a transition from exponential to power-law decay in the subcritical phase.
This contrasts with the classical Bernoulli site percolation in which exponential decay holds everywhere outside the critical point \cite{Aizenman87, Duminil-Copin16, Menshikov86} and raises the question whether the phase transition for Bernoulli hyperplane percolation is sharp in the sense that the expected size of the cluster containing a vertex is finite in the whole subcritical regime.
For Bernoulli site percolation sharpness is an immediate consequence of exponential decay.

Let us now briefly comment on some related results obtained for percolation models presenting infinite-range correlations along columns.
One of these models is the so called Winkler's percolation \cite{winkler_2000} for which a power-law decay as in \eqref{e:power-law} has been proved by G\'acs \cite{gacs_2000} whenever the model is supercritical.
For another model called corner percolation, although all the connected components are finite almost surely, Pete \cite{G.Pete} has obtained a power-law lower bound for $\PP_{\mathbf{p}}(o \leftrightarrow \partial{B}(K))$ .
A variation of Bernoulli line percolation was studied in \cite{Grassberger17_2}.
In this paper, only columns that extend along a single direction are removed and Bernoulli line percolation is performed on the remaining graph.
Here \eqref{e:power-law} holds in some parts of the subcritical phase and throughout the whole supercritical phase.

We finish this section presenting a brief overview of the remainder of the paper.
In Section \ref{sec:notation} we introduce some of the notation that will need.
Section \ref{sec:subcritical_phase} is devoted to the study of the subcritical phase, that is, the regime in which infinite clusters occur with null probability.
Lemma \ref{lemma nao perc 1} identifies values of the parameters which falls inside the subcritical phase and thus implies the second assertion in Theorem \ref{t:phase_transition}.
We also present other results that add more information about the subcritical phase including bounds on the parameters that guarantee exponential decay of correlations (Remark \ref{remark:exp_dec}).
In Section \ref{sec:supercritical_phase} we prove the existence of the supercritical phase, that is, the regime in which there exists at least one infinite open cluster with probability one.
This corresponds to the first assertion in Theorem \ref{t:phase_transition}.
In Section \ref{sec:polynomial_decay} we present the proof of Theorem \ref{t:poly_dec_1} and show how to modify it in order to obtain a proof for Theorem \ref{theorem power law projetando em k}.
These are perhaps the most interesting results in our work since they highlight the presence of power-law decay of connectivity in some regimes and show that the transition from the subcritical to the supercritical phase is more delicate than that exhibited by ordinary percolation models with finite range dependencies.

\subsection{Notation}
\label{sec:notation}

In this section we make precise the notation and definitions used in the previous section and introduce some further notation that will be used in the remainder of the paper.

The $n$-dimensional Euclidean lattice (here called simply the $\mathbb{Z}^n$-lattice) is the pair $\mathbb{Z}^n = (V(\mathbb{Z}^n), E(\mathbb{Z}^n))$ whose  vertex set $V(\mathbb{Z}^n)$ is composed of vectors $x=(x_1,\ldots, x_n) \in \mathbb{R}^n$ having integer coordinates $x_i$ and $E(\mathbb{Z}^n)$ is the set of pairs of vertices in $V(\mathbb{Z}^n)$ lying at Euclidean distance one from each other, called edges (or bonds).
Vertices $x \in V(\mathbb{Z}^n)$ will also be called sites.
We abuse notation using $\mathbb{Z}^n$ to refer both to the $\mathbb{Z}^n$-lattice and to its set of vertices.
We denote $\|x\|=\sum_{i=1}^n |x_i|$ the $l_1$-norm of  $x=(x_1,\ldots, x_n) \in\ZZ^n$ .

The vertex $o =(0,\ldots,0) \in \mathbb{Z}^n$ will be called the origin.
Note that it also belongs to each one of the $\mathbb{Z}^k_I$.
We write $B(K):=[-K,K]^n\cap \ZZ^n$ for the $l_\infty$-ball of radius $K$ centered at $o$.
For a given $I \in \mathcal{I}(k;n)$, we will also use $B(K)$ instead of $\pi_I(B(K))$ for the corresponding box contained in the hyperplanes $\mathbb{Z}^k_I$ (recall the definitions of the index set $\mathcal{I}(k;n)$ in \eqref{e:index_set} and of the $k$-dimensional coordinate hyperplanes $\mathbb{Z}^k_I$ in \eqref{eq:zkn}).

Consider $\Omega = \{0,1\}^{\mathbb{Z}^n}$ endowed with the canonical sigma-field $\mathcal{F}$ generated by the cylinder sets.
A probability measure $\mu$ on $(\Omega, \mathcal{F})$ is called a site percolation on $\mathbb{Z}^n$.
Any random element $(\omega(v))_{v\in \mathbb{Z}^n}$ which is distributed as $\mu$ is also called a percolation process in $\mathbb{Z}^n$.
For $p \in [0,1]$, we denote $\mathbb{P}_p$ be the probability measure in $(\Omega, \mathcal{F})$ under which the projections $(\omega(v))_{v\in \mathbb{Z}^n}$ are i.i.d.\ Bernoulli random variables of mean $p$.
This is the so-called Bernoulli site percolation on $\mathbb{Z}^n$ with parameter $p$.

Fix integers $n \geq 3$ and $2 \leq k \leq n-1$ and let $p_I \in [0,1]$ for each $I \in \mathcal{I}(k;n)$.
Consider $\Omega_I =\{0,1\}^{\mathbb{Z}^k_I}$ endowed with the canonical sigma-field $\mathcal{F}_I$. 
The definition of Bernoulli site percolation with parameter $p_I$ extends to $\mathbb{Z}^k_I$ naturally yielding the measures $\mathbb{P}_{p_I}$ on $(\Omega_I, \mathcal{F}_I)$.
The probability measure $\mathbb{P}_{\mathbf{p}}$ which was defined below \eqref{eq:def_omega}, is the unique measure in $(\Omega, \mathcal{F})$ satisfying
\begin{equation}
\label{e:relation_measures}
\mathbb{P}_{\mathbf{p}} = \big(\otimes_{I \in \mathcal{I}_{k,n}} \mathbb{P}_{p_I} \big) \circ \omega^{-1} \end{equation}
where $\omega$ is defined in \eqref{eq:def_omega}.
We will denote $\EE_\mathbf{p} ( \cdot )$ the expectation with respect to $\mathbb{P}_{\mathbf{P}}$.

Let $G$ be either $\mathbb{Z}^n$ or $\mathbb{Z}^k_I$ for some $I \in \mathcal{I}$.
Given $\eta = (\eta(x))_{G} 
\in \{0,1\}^{G}$, we say that a site $x \in G$ is $\eta$-open when $\eta(x) =1$.
Otherwise $x$ is said $\eta$-closed.
A site $x \in \mathbb{Z}^n$ is said $\omega_I$-open if $\pi_I(x)\in\ZZ_I^k$ is $\omega_I$-open, i.e., if $\omega_I(\pi_I(x))=1$. 
Otherwise, $x$ is said $\omega_I$-closed. 
Since $x$ is $\omega$-open if and only if it is $\omega_I$-open for all $I\in\mathcal{I}(k,n)$ and since the percolation processes $\omega_I$ are independent we have 
\begin{equation}
\label{e:exp_p}
\EE_\mathbf{p} (\omega(o)) = \PP_{\mathbf{p}}(\omega(x)=1)=\prod_{I\in\mathcal{I}}p_I.
\end{equation}
Therefore one can show that
\begin{equation}
\label{e:box_open}
\PP_\mathbf{p}\big(\mbox{all sites in }B(R)\mbox { are open} \big) = \big[\EE_\mathbf{p}(\omega(o))\big]^{(2R+1)^k}
\end{equation}
which is similar to the Bernoulli site percolation case where the exponent $(2R+1)^k$ has to be replaced by $(2R+1)^n$.

Let $G$ be either $\mathbb{Z}^n$ or $\mathbb{Z}^k_I$ for some $I \in \mathcal{I}$.
A path in $G$ is either a finite set $\Gamma = \{x_0, x_1, \ldots, x_m\}$ or an infinite set $\Gamma = \{x_0, x_1, x_2 \ldots\}$ such that $x_i \in G$ for all $i$, $x_i \neq x_j$ whenever $i\neq j$ and $||x_i - x_{i-1}||=1$ for all $i=1, \ldots, m$.
For $\eta=(\eta(x))_{x\in G} \in \{0,1\}^G$ we denote $\mathcal{V}_{\eta} = \mathcal{V}_{\eta}(G)=\{x\in G : \eta(x)=1\}$. 
For such $\eta$ we say that $x,y\in\mathcal{V}_{\eta}$ are connected and write $x \leftrightarrow y$ if there exists a path composed exclusively of $\eta$-open sites that starts at $x$ and finishes at $y$.
Otherwise, we write $x \nleftrightarrow y$.
For $A \subset \mathbb{Z}^n$, we write $x \leftrightarrow A$ if $x\leftrightarrow y$ for some $y\in A$. 
We say that $\mathcal{G}\subset \mathcal{V}_{\eta}$ is a connected component (or a cluster) of $\mathcal{V}_{\eta}$ when every pair of sites $x,y\in\mathcal{G}$ is such that $x\leftrightarrow y$.
In addition, for $x\in G$ we denote by $\mathcal{V}_\eta(x) = \mathcal{V}_\eta(x; G)$ the maximal connected component in $\mathcal{V}_\eta$ containing $x$, that is, $\mathcal{V}_\eta(x; G) = \{y \in G \colon  x\leftrightarrow y \text{ in $\eta$}\}$.
We say that a site $x$ belongs to an infinite connected component in $\mathcal{V}_\eta$ and denote it $x \leftrightarrow \infty$ if $\# \mathcal{V}_{\eta}(x) = \infty$.

We say that the Bernoulli $(n,k)$-hyperplane percolation exhibits a non-trivial phase transition if there exists $\mathbf{p}=(p_I)_{I \in \mathcal{I}}$ with $p_I>0$ for all $I\in\mathcal{I}$ and $\mathbf{q}=(q_I)_{I \in \mathcal{I}}$ with $q_I<1$ for all $I\in\mathcal{I}$, such that $\PP_{\mathbf{p}}(o\leftrightarrow \infty)=0$ and $\PP_{\mathbf{q}}(o\leftrightarrow \infty)>0$.
The set of all $\mathbf{p}$ for which $\PP_{\mathbf{q}}(o\leftrightarrow \infty)>0$ is called the supercritical phase whereas the set of all $\mathbf{p}$ for which $\PP_{\mathbf{p}}(o\leftrightarrow \infty)=0$ is the subcritical phase.

For $x\in\ZZ^k_I$ we denote by 
\begin{equation}
\label{eq:affine_hyperplanes}
\mathcal{P}_I(x):=\pi_I^{-1}(x)= \{z \in \ZZ^n \colon\, \pi_I(z)=x\}
\end{equation} 
the pre-image of $x$ under $\pi_I$, so that $\{ \mathcal{P}_I(x) \colon x\in \mathbb{Z}^k_I \}$ foliates $\mathbb{Z}^n$ into disjoint `parallel $(n-k)$-dimensional afine hyperplanes'.
Observe that
\[
\text{
$\inf\big\{\|v-w\|:v\in \mathcal{P}_I(x),w\in \mathcal{P}_I(y)\big\}=1$
if and only if $\|x-y\|=1$.}
\]
Let $I\in\mathcal{I}(k;n)$. The graph $\mathcal{H}=\mathcal{H}(I)$ with vertices $V(\mathcal{H}):=\{\mathcal{P}_I(x):x\in\ZZ^k_I\}$, and with edges linking pairs of vertices $\mathcal{P}_I(x)$ and $\mathcal{P}_I(y)$ satisfying
\[
\inf\{||v-w||:v\in \mathcal{P}_I(x),w\in \mathcal{P}_I(y)\}=1
\]
is called a $\ZZ^k$-decomposition of $\ZZ^n$.
Notice that $\mathcal{H}$ is isomorphic to $\ZZ^k$.
 
For a fixed $x\in\ZZ^k_I$, the projection $\pi_{[n]\setminus I}:\ZZ^n\to\ZZ_{[n]\setminus I}^{n-k}$ maps $\mathcal{P}_I(x)$ isomorphically to $\ZZ_{[n]\setminus I}^{n-k}$ which is, in turn, isomorphic to $ \ZZ^{n-k}$.
Thus for each $v\in \ZZ_{[n]\setminus I}^{n-k}$, there exists a unique $u\in \mathcal{P}_I(x)$ such that $\pi_{[n]\setminus I}(u)=v$. 
We say that $v$ is $\mathcal{P}_I(x)$-closed if there exists $J\in \mathcal{I}(n;k) \setminus \{I\}$ for which $u$ is $\omega_J$-closed.
Otherwise we say that $v$ is $\mathcal{P}_{I}(x)$-open.
Observe that if $x$ and $y$ are different vertices in $\mathbb{Z}^k_I$ and $v \in \mathbb{Z}^{n-k}_{[n]\setminus I}$ we might have that $v$ is $\mathcal{P}_I(x)$-open and $\mathcal{P}_I(y)$-closed.

We say that $T\subset \ZZ^n$ surrounds the origin if there exists a partition of $\ZZ^n \setminus T = A \cup B$ such that:
\begin{eqnarray}
 \text{$A$ is connected, $o\in A$ and $\# A <\infty$;} \label{e:surround_1}\\
\text{$\inf\{\|a-b\|:a\in A, b\in B \}\geq 2$.}\label{e:surround_2}
\end{eqnarray}
Similar definitions can be made replacing $\mathbb{Z}^n$ by any of the $\mathbb{Z}^k_I$.
A useful fact that we will use below it that $\mathcal{V}_{\omega}(o)$ is finite if and only if there exists $T\subset \ZZ^n$ that surrounds the origin and whose sites are all $\omega$-closed (and similarly for $\mathcal{V}_{\omega_I}$).

\section{The existence of a subcritical phase}
\label{sec:subcritical_phase}

This section is dedicated to the existence of a subcritical phase.
Indeed we show that Bernoulli hyperplane percolation does not present infinite connected components a.s.\ when some of the parameters of $\mathbf{p}$ are sufficiently small.
This corresponds to the second assertion in Theorem \ref{t:phase_transition} which is a consequence of Lemma \ref{lemma nao perc 1} below.
Roughly speaking, Lemma \ref{lemma nao perc 1} asserts that the probability that a given site belongs to an infinite open cluster vanishes as soon as a single parameter $p_I$ is taken subcritical and at least other (well-chosen) $n-k$ parameters do not equal $1$.
In order to get the same conclusion, Lemma \ref{lemma nao perc 2} requires that $n/k$ parameters are subcritical regardless of the fact that the other parameters can even be equal to $1$.
Remark \ref{remark:exp_dec} contains the sketch of an argument showing that if we get sufficiently many subcritical parameters then actually exponential decay holds (hence infinite connected components cannot exist $a.s.$). 
We begin with the following deterministic result which will also be useful in Section \ref{sec:polynomial_decay} when we present a proof of Theorem \ref{t:poly_dec_1}.

\begin{lemma}
\label{lemma nao perc} 
Let $\omega$ be as in  \eqref{eq:def_omega} and fix  $I\in\mathcal{I}(k;n)$.
Assume that the two following conditions hold:

\noindent i) The cluster $\mathcal{V}_{\omega_I}(o;\ZZ_I^k)$ is finite;\\
\noindent ii) There exists $T\subset \ZZ_{[n]\setminus I}^{n-k}$ that surrounds the origin in $\ZZ_{[n]\setminus I}^{n-k}$ and such that every $v\in T$ is $\mathcal{P}_I(x)$-closed for every $x\in\mathcal{V}_{\omega_I}(o;\ZZ_I^k)$.

\noindent Then $\mathcal{V}_\omega(o; \mathbb{Z}^n)$ is finite.
\end{lemma}

\begin{proof}
\noindent Assume that there exists an infinite path $\{o=z_1,z_2,\ldots\}\subset \ZZ^n$  starting at the origin and composed of $\omega-$open sites only.
Let $T$ be as in Condition ii) and $A\subset\ZZ_{[n]\setminus I}^{n-k}$ be the corresponding set given as \eqref{e:surround_1} and \eqref{e:surround_2}. 
We claim that $z_i$ satisfies
\begin{equation}
\label{e:proj_path}
\text{$\pi_I(z_{i})\in \mathcal{V}_{\omega_I}(o,\ZZ_I^k)\,\,\,\,$ and $\,\,\,\,\pi_{[n]\setminus I}(z_{i})\in A\,\,\,\,$ for all $i=1,2,\ldots$}
\end{equation}
Since $\mathcal{V}_{\omega_I}(o;\ZZ_I^k)$ and $A$ are finite this would contradict the fact that all the $z_i$'s are distinct. 

Since $z_1=o$, \eqref{e:proj_path} holds for $i=1$.
Now assume that \eqref{e:proj_path} holds for some $i \geq 1$.
Since $\|z_i-z_{i+1}\|=1$, either we have $\|\pi_{[n]\setminus I}(z_{i+1})- \pi_{[n]\setminus I}(z_i)\|=1$ and $\|\pi_I(z_{i+1})-\pi_I(z_i)\|=0$ or else $\|\pi_{[n]\setminus I}(z_{i+1})-\pi_{[n]\setminus I}(z_i)\|=0$ and $\|\pi_I(z_{i+1})-\pi_I(z_i)\|=1$.
In the first case, let  $x= \pi_I(z_i) = \pi_I(z_{i+1})$.
We have $\pi_I(z_{i+1})\in \mathcal{V}_{\omega_I}(o,\ZZ_I^k)$. 
Moreover, since $\|\pi_{[n]\setminus I}(z_{i+1})-\pi_{[n]\setminus I}(z_i)\|=1$,
we have $\pi_{[n]\setminus I} (z_{i+1}) \in A \cup T$.
But $\pi_{[n]\setminus I} (z_{i+1})$ is not $\mathcal{P}_I(x)$-closed, so $\pi_{[n]\setminus I} (z_{i+1}) \notin T$, therefore we must have $\pi_{[n]\setminus I} (z_{i+1}) \in A$.
In the second case, $z_i\in \mathcal{P}_I(x)$ and $z_{i+1}\in \mathcal{P}_I(y)$, where   $\|x-y\|=1$ and $x\in \mathcal{V}_{\omega_I}(o;\ZZ_I^{k})$ thus, since $\omega_I(y)=1$ we must have that $y \in \mathcal{V}_{\omega_I}(o; \mathbb{Z}^k_I)$.
Also $\pi_{[n]\setminus I}(z_i)=\pi_{[n]\setminus I}(z_{i+1})$ and hence $\pi_{[n]\setminus I} (z_{i+1}) \in A$.
Therefore, \eqref{e:proj_path} follows by induction.
\qed
\end{proof}

We use Lemma \ref{lemma nao perc} in order to prove the following result  that settles the existence of a subcritical phase proving the second assertion in Theorem \ref{t:phase_transition}.

\begin{lemma}
\label{lemma nao perc 1} 
Assume that  $p_I< p_c(\ZZ^k)$ for some $I\in \mathcal{I}(k;n)$ and that $p_{J_i}<1$ for the $n-k$ distinct $J_i\in \mathcal{I}(k;n)$ such that $\# (I\cap J_1\cap\ldots\cap J_{n-k}) = k-1$. 
Then $\mathcal{V}_{\omega}(o; \mathbb{Z}^n)$ is finite $\mathbb{P}_{\mathbf{p}}$-$a.s.$
\end{lemma}
\begin{proof} The proof is divided into $2$ cases:
\medskip

\noindent \textbf{First case: $k=n-2$, $k\geq 2$:}

Let us assume for simplicity that $I=\{1,\ldots,k\}$, $J_1=\{1,\ldots,k-1,k+1\}$ and $J_2=\{1,\ldots,k-1,k+2\}$
(thus $p_I<p_c(\ZZ^k)$ and $p_{J_1},p_{J_2}<1$).
Let $\mathcal{H}=\mathcal{H}(I)$ be the $\ZZ^k$-decomposition of $\ZZ^n$ associated to $I$. 
Then $\mathcal{H}$ is isomorphic to $\ZZ^k$ and each site of $\mathcal{H}$ is isomorphic to $\ZZ^2$. 
Since $p_I<p_c(\ZZ^k)$,  there exists $a.s.$ a (random) non-negative integer $N$  such that $\mathcal{V}_{\omega_I}(o;\ZZ_I^k)\subset B(N)$. 
In particular, Condition i) in Lemma \ref{lemma nao perc} holds $a.s.$ and all  we need to show is that Condition ii) holds a.s.\ on the event $[\mathcal{V}_{\omega_I}(o;\ZZ_I^k)\subset B(N)]$ for each fixed $N$.

To this end, first recall that for each $x\in\ZZ_I^k$, $\mathcal{P}_I(x)=\{z\in \ZZ^n: \pi_I(z)=x\}$. 
Since $p_{J_1}<1$, the Borel-Cantelli Lemma guarantees that, almost surely, there exists $x_{k+1}^*\in\NN$ such that $\omega_{J_1}(x_1,\ldots,x_{k-1},x_{k+1}^*)=\omega_{J_1}(x_1,\ldots,x_{k-1},-x_{k+1}^*)=0$, for all $(x_1,\ldots,x_{k-1})\in[-N,N]^{k-1}\cap\ZZ^{k-1}$. 
This implies that for each $x=(x_1,\ldots,x_k)\in [-N,N]^k \cap \mathbb{Z}^k$
$$\omega(x_1,\ldots,x_k,x_{k+1}^*,x_{k+2})=\omega(x_1,\ldots,x_k,-x_{k+1}^*,x_{k+2})=0, \, \forall \, x_{k+2}\in\ZZ.$$
In other words, for each $x\in B(N)\subset \mathbb{Z}^k_I $ the set of $\mathcal{P}_I(x)$-closed sites of $\pi_{[n]\setminus I}(\mathcal{P}_I(x))\subset \ZZ^2_{[n]\setminus I}$ contains the lines $\{(0,\ldots, 0, x_{k+1}^*,s):s\in\ZZ\}$ and $\{(0,\ldots,0,-x_{k+1}^*,s):s\in\ZZ\}$.
Similarly, we find $x_{k+2}^*$ such that for each $x=(x_1,\ldots,x_k)\in [-N,N]^k \cap \mathbb{Z}^k$
$$\omega(x_1,\ldots,x_k,x_{k+1},x_{k+2}^*)=\omega(x_1,\ldots,x_k,x_{k+1},-x_{k+2}^*)=0, \, \forall \, x_{k+1}\in\ZZ.$$
Hence for each $x\in B(N) \subset \mathbb{Z}^k_I$, the set of $\mathcal{P}_I(x)$-closed sites of $\pi_{[n]\setminus I}(\mathcal{P}_I(x))\subset \ZZ^2$ contains the lines $\{(0,\ldots,0,t,x_{k+2}^*):t\in\ZZ\}$ and $\{(0,\ldots,0,-t_{k+2}^*):t\in\ZZ\}$. 
Therefore, for each fixed $N$, Condition ii) in Lemma \ref{lemma nao perc} holds $a.s.$ in the event $\mathcal{V}_{\omega_I}(o;\ZZ_I^k)\subset[-N,N]^k$ with $T\subset \ZZ^2_{[n]\setminus I}$ the rectangle delimited by the lines $\{(0, \ldots, 0, x_{k+1}^*,s):s\in\ZZ\}$, $\{(0, \ldots, 0, -x_{k+1}^*,s):s\in\ZZ\}$, $\{(0,\ldots,0,t,x_{k+2}^*):t\in\ZZ\}$ and $\{(0,\ldots,0,t,-x_{k+2}^*):x\in\ZZ\}$. 
This completes the proof in the case $n=k+2$.
\medskip

\noindent \textbf{The general case $n=k+l$, $k,l\geq2$:}

Let us assume for simplicity that $I=\{1,\ldots,k\}$ and $J_i=\{1,\ldots,k-1,k+i\}$ for $1\leq i \leq l$. 
Let $\mathcal{H}=\mathcal{H}(I)$ be the $\ZZ^k$-decomposition of $\ZZ^n$ associated to $I$. 
Then $\mathcal{H}$ is isomorphic to $\ZZ^k$ and each site in $\mathcal{H}$ is isomorphic to $\ZZ^l$. Similarly as above, condition $p_I<p_c(\ZZ^k)$ implies that Condition i) in Lemma \ref{lemma nao perc} holds $a.s.$ 
Hence we only need to show that $p_{J_i}<1$ for $1\leq i \leq l$ implies that Condition ii) in the same Lemma holds $a.s.$ 
Similarly as above, for each $1\leq i\leq l$,
the family of random variables $\{\omega_{J_i}(x_1,\ldots,x_{k-1},x_{k+i}): -N\leq x_1,\ldots,x_{k-1}\leq N, x_{k+i}\in\ZZ\}$ are independent, and hence by the Borel-Cantelli Lemma there exists $a.s.$ a set of nonnegative integers $\{x_{k+i}^*\}_{i=1}^l$ such that $\omega_{J_i}(x_1,\ldots,x_{k-1},x_{k+i}^*)=\omega_{J_i}(x_1,\ldots,x_{k-1},-x_{k+i}^*)=0$, for all $-N\leq x_1,\ldots,x_{k-1}\leq N$ and all $i=1,\ldots,l$. 
This implies that, for every $i=1,\ldots, l$, the set of $\mathcal{P}_I(x)$-closed sites of $\pi_{[n]\setminus I}(\mathcal{P}_I(x))$ contains the hyperplane 
\[
T_i :=\big\{z=(0, \ldots, 0, z_{k+1},\ldots,z_{k+l}) \in \ZZ^l_{[n]\setminus I} \colon \, 
z_{k+i} = \pm x_{k+1}^*\big\}.
\]
This shows that Condition ii) holds $a.s.$ with
\[
T:=\{0\} \times \cdots \times \{0\} \times \partial \big([-x_{k+1}^*,x_{k+1}^*]\times\ldots\times [-x_{k+l}^*,x_{k+l}^*]\big) \cap \ZZ^l_{[n]\setminus I}.
\]
\qed
\end{proof}

\begin{remark}[Proof of the second assertion in Theorem \ref{t:phase_transition}]
\label{remark:proof_phase_tran_2}
It follows directly from the statement of Lemma \ref{lemma nao perc 1}  that when the parameters $p_I$ are sufficiently small (e.g.\ if they all belong to  the interval $(0,p_c(Z^k))$) then percolation does not occur.
Notice, however, that Lemma \ref{lemma nao perc 1} provides much more detail on the location of the subcritical phase in the space of parameters.
\end{remark}

The next result also implies the existence of the subcritical phase.
Strictly speaking, it only holds in the particular setting when $k$ divides $n$ and, although it will not be used it in the remainder of the paper, we decided to include it here because it adds some further information to the phase diagram in this specific setting.
Its proof also uses Lemma \ref{lemma nao perc}.

\begin{lemma}
\label{lemma nao perc 2}
Assume that $k$ divides $n$, and let $\omega$ be a Bernoulli $(n,k)$-hyperplane percolation process. 
Let $I_1,\ldots,I_{n/k}\in\mathcal{I}(k;n)$ be a partition of $[n]$. 
If for each $1\leq j\leq n/k$ we have $p_{I_j}<p_c(\ZZ^k)$, then $\mathcal{V}_{\omega}(o; \mathbb{Z}^n)$ is finite $\mathbb{P}_{\mathbf{p}}$-$a.s.$
\end{lemma}
\begin{proof} 
Write $n=lk$. 
Without loss of generality, assume that for each $1\leq j \leq l$, $I_j=\{(j-1)k+1,\ldots,jk\}$. 
We will argue by induction on $l \geq 2$. 

We begin fixing $l=2$. 
Let $\mathcal{H}=\mathcal{H}(I_1)$ be the $\ZZ^k$-decomposition of $\ZZ^{2k}$ corresponding to $I_1$. 
Then each site $\mathcal{P}_{I_1}(x)$ of $\mathcal{H}$ is isomorphic to $\ZZ^k$. 
We want to verify that Conditions i) and ii) in Lemma \ref{lemma nao perc} hold $a.s.$ with $I = I_1$.
Since $p_{I_1}<p_c(\ZZ^k)$, the cluster $\mathcal{V}_{\omega_{I_1}}(o, \mathbb{Z}^{k_{I_1}})$ is finite $a.s.$, hence Condition i) holds $a.s.$
Let us condition on $\mathcal{V}_{\omega_{I_1}}(o, \mathbb{Z}^{k_{I_1}})$ and show that Condition ii) also holds $a.s.$

Since $l=2$ we have $\mathbb{Z}^{n-k}_{[n]\setminus I} = \mathbb{Z}^k_{I_2}$. 
Now, since $p_{I_2}<p_c(\ZZ^k)$, the cluster $\mathcal{V}_{\omega_{I_2}}(o, \mathbb{Z}^{k}_{I_2})$ is finite a.s.\
Hence, almost surely, there exists $T \subset \mathbb{Z}^k_{I_2}$ that surrounds the origin and whose sites are $\omega_{I_2}$-closed.
In particular, they are $\mathcal{P}_{I_1}(x)$-closed for every $x \in \mathcal{V}_{\omega_1}(o, \mathbb{Z}^{k}_{I_1})$.
This shows that Condition ii) in Lemma \ref{lemma nao perc} holds $a.s.$ with $I = I_1$.  
This concludes the proof for the case $l=2$. 

Now, assume  that the result holds for some $l\in\NN$ and let $n=(l+1)k$.
Since $p_{I_1}<p_c(\ZZ^k)$, Condition i) in the same lemma holds $a.s.$ with $I = I_1$.
Let us condition on the cluster $\mathcal{V}_{\tilde{\omega}_1}(o;\mathbb{Z}^{l}_{I_1})$ and show that Condition ii) holds.

Since $\mathbb{Z}^{lk}_{[n]\setminus I_1}$ is isomorphic to $\mathbb{Z}^{lk}$ we can define naturally Bernoulli $(lk,k)$-hyperplane percolation processes on it.
In fact, one can show that $(\tilde{\omega}(v))_{v \in \mathbb{Z}^{lk}_{[n]\setminus I_1}}$ defined as
\[
\tilde{\omega}(v)=\prod_{\substack{J\in\mathcal{I}(k;n) \\ J\cap I_1=\varnothing }}\omega_{I}(\pi_J(v))
\]
is indeed such a percolation process.
Since for each $2\leq j \leq l+1$ we have $p_{I_j}<p_c(\ZZ^k)$, one can use the  induction hypothesis to obtain that $\mathcal{V}_{\tilde{\omega}}(o;\mathbb{Z}^{lk}_{[n]\setminus I_1})$ is finite almost surely.
Thus, almost surely, there exists a set $T \subset \mathbb{Z}^{lk}_{[n]\setminus I_1}$ that surrounds the origin in $\mathbb{Z}^{lk}_{[n]\setminus I_1}$ and whose sites are all $\tilde{\omega}$-closed.
Therefore, for each $v \in T$, there exists $J \in \mathcal{I}(k;n)\setminus I_1$ such that $\omega_J(v) =0$ which means that $v$ is $\mathcal{P}_{\omega_{I_1}}(x)$-closed for every $x \in \mathbb{Z}^k_{I_1}$, in particular, for every $x \in \mathcal{V}_{\omega_{I_1}}(o;\mathbb{Z}^k_{I_1})$.
This establishes Condition ii).
The result follows by induction.
\qed
\end{proof}

We close this section presenting a sketch to a proof for the existence of regimes in which the connectivity decay is exponential.
\begin{remark}[Exponential decay] 
\label{remark:exp_dec}
Assume that at least $\binom{n-1}{k}+1$ of the parameters $I \in \mathcal{I}(k;n)$ satisfy $p_I < p_c(\mathbb{Z}^k)$.
If the event $\big[o \leftrightarrow \partial{B(K)}\big]$ holds for some integer $K>1$ then there must be at least one site $x=(x_1,\ldots, x_n) \in \partial{B(K)}$ for which $[o \leftrightarrow x]$.
Such a $x$ has at least one coordinate, say $x_{i_o}$, with $x_{i_o} = K$.
By continuity of the projections into the coordinate planes, the events $\big[o \leftrightarrow \pi_J\big(\partial{B(K)}\big)\big]$ must occur for all the indices $J \in \mathcal{I}(k;n)$ containing $i_o$.
This amounts for exactly $\binom{n-1}{k-1}$ indices, hence by our assumption, there must be at least one of these indices $J$ for which $p_J < p_c(\mathbb{Z}^k)$.
This implies that $\mathbb{P}_{p_J} \big(o \leftrightarrow \pi_J\big(\partial{B(K)}\big)\big)$ decays exponentially fast \cite{Aizenman87, Menshikov86} (see also \cite{Duminil-Copin16} for a more elementary proof).
Therefore, $\mathbb{P}_{\mathbf{p}} \big( o \leftrightarrow \partial{B(K)} \big)$ must also decay exponentially fast.
\end{remark}

\section{The existence of a supercritical phase}
\label{sec:supercritical_phase}

Our aim in this section is to prove that configurations in Bernoulli hyperplane percolation contain infinite connected components almost surely as soon as the parameters of $\mathbf{p}$ are large enough.
According to \eqref{e:exp_p} this is equivalent to $\mathbb{E}_{\mathbf{p}}\big(\omega(o)\big)$ being sufficiently close to $1$.
This is the content of the next result which  readily implies the first assertion in Theorem \ref{t:phase_transition}.

\begin{theorem}\label{theorem existence of percolation} Let $2\leq k \leq n-2$ and $\omega$ be as in \eqref{eq:def_omega}. 
If $\EE_\mathbf{p} (\omega(o))$ is sufficiently  close to $1$, then $\PP_{\mathbf{p}}(o\longleftrightarrow \infty)>0$.
\end{theorem}

\begin{remark}[Proof of the first assertion in Theorem \ref{t:phase_transition}]
\label{remark:proof_phase_tran_1}
By \eqref{e:exp_p}, $E_{\mathbf{p}}(\omega(o))$ can be arbitrarily close to $1$ provided that all the $p_I$ are sufficiently large (all of them still smaller than $1$).
Therefore, the first assertion in Theorem \ref{t:phase_transition} follows readily from Theorem \ref{theorem existence of percolation}.
\end{remark}

In percolation, such a result is usually obtained with the help of Peierls-type arguments, that is, by restricting the process to the plane $\mathbb{Z}^2$ and showing that, as long as the control parameter are made large enough, large closed sets surrounding the origin in $\mathbb{Z}^2$ are very unlikely.
In our case, restricting the model to $\mathbb{Z}^2$ is not useful since the plane will be disconnected into finite rectangles.
We therefore replace $\mathbb{Z}^2$ with an subgraph resembling a plane that is inclined with respect to the coordinate axis in order to gain some independence.

For that we will use the following auxiliary result whose proof relies on elementary arguments and is presented in the Appendix.
\begin{lemma}
\label{lemma construcao da base da rede} 
Let $2\leq k \leq n-1$.
There exist orthogonal vectors $w_1$ and $w_2$ in $\ZZ^n$ with $\|w_1\|=\|w_2\|$ such that the linear application $A:\RR^2\to\RR^n$ given by $A(x,y)=xw_1+yw_2$ satisfies the following properties:\\
i.  For every $I\in \mathcal{I}(k;n)$ the mapping $\pi_I\circ A:\ZZ^2\to\ZZ^{\#I}$ is injective;\\
ii. There exists a constant $c=c(w_1,w_2)>0$ such that for every $I \in \mathcal{I}(k;n)$, and every $u$ and $v$ in $\RR^2$, $\|\pi_I(Au-Av)\|\geq c \|u-v\|.$
\end{lemma}

For the rest of this section we fix the dimension $n$, the vectors and $w_1$ and $w_2$, and the corresponding linear application $A$ as in Lemma \ref{lemma construcao da base da rede}.
We define $\mathcal{G}_0 = A(\ZZ^2)$.
By Condition i.\ in Lemma \eqref{lemma construcao da base da rede} for every $2\leq k\leq n-1$, the Bernoulli $(n,k)$-hyperplane percolation  process $\omega$ restricted to $\mathcal{G}_0$ has i.i.d.\ states
\textit{i.e.}, the process $\eta_0:=\{\omega(A v)\}_{v\in\ZZ^2}$ is a standard Bernoulli site percolation process in $\ZZ^2$ with parameter $p=\EE_{\mathbf{p}}(\omega(o))$.
Thus, for $p$ close to $1$, $\eta_0$ has an infinite cluster $a.s.$ whose image under $A$ is also an infinite set composed of $\omega$-open sites in $\mathbb{Z}^n$.
However, $\mathcal{G}_0$ is not necessarily a connected subgraph of $\ZZ^n$ and thus we did not prove that this set is indeed and infinite open cluster.
To fix this issue we will add sites to $\mathcal{G}_0$ in such a way to guarantee that we get a connected subgraph $\mathcal{G}\subset\ZZ^n$.
Now the family of random variables $\{\omega(x)\}_{x\in\mathcal{G}}$ may no longer be independent. 
However it will still dominate an independent family as long as the parameters $p_I$ are large enough.
This will allow us to find an infinite cluster in $\mathcal{G}$ a.s.

\subsection{Construction of the graph $\mathcal{G}$}
Let $w_1$ and $w_2$ be as above and denote $w_1=(\alpha_1,\ldots,\alpha_n)$ and $w_2=(\beta_1,\ldots,\beta_n)$. 
Let $p_0=q_0=o$ and define inductively for $1\leq j\leq n$:
\begin{align*}
p_j&=p_{j-1}+\alpha_je_j,\\
q_j&=q_{j-1}+\beta_je_j.
\end{align*}
Given $u,v\in\ZZ^n$ such that $u-v=ze_j$ for some $z\in\ZZ$, denote
$[u,v]=\{w\in \mathbb{Z}^n \colon\, w= u+l(z/|z|)e_j, \, l=0, \ldots, |z|\}$
(if $z = 0$, then $u=v$ so set $[u,v] = \{u\}$).
Let $\Gamma(0,0):=\bigcup_{j=1}^n [p_{j-1},p_j]\cup[q_{j-1},q_j]$ and $\Gamma(x,y)=\{A(x,y)+v:v\in\Gamma(0,0)\}$ which contains a path that starts at $A(x,y)$ and ends at $A(x+1,y)$ and another that starts at $A(x,y)$ and ends at $A(x,y+1)$.
Therefore, if we denote
\[
\mathcal{G}:=\bigcup_{(x,y)\in\ZZ^2}\Gamma(x,y),
\]
then, when regarded as a subgraph of the $\mathbb{Z}^n$ lattice, $\mathcal{G}$ is connected.

As mentioned above, we will study the percolation process restricted to $\mathcal{G}$ which we hope will dominate a supercritical percolation process.
In implementing these ideas, the standard results of Liggett, Schonmann and Stacey \cite{Ligget} are very useful.
Before we state it precisely, let us give the relevant definitions.
 
A random element $\big(f(x)\big)_{x\in\ZZ^n} \in \{0,1\}^{\mathbb{Z}^n}$ is said of class $C(n,\chi,p)$ if for every $x\in\ZZ^n$ and $S\subset\ZZ^n$ such that $\inf\{\|a-x\|:a\in S\}\geq \chi$,
we have $\PP\big(f(x)=1|(f(a))_{a\in S}\big)\geq p$.
Such elements appear naturally when performing one-step renormalization arguments.
We are ready to state a result that will help to control the process restricted to $\mathcal{G}$ and will also be used in Section \ref{sec:polynomial_decay}. 
It consists of a rephrasing of the part of the statement of Theorem 0.0 in \cite{Ligget} that serves our purposes.
\begin{theorem}[Theorem 0.0 in \cite{Ligget}]
\label{remark Ligget}
For every $\rho>0$ and $\chi$ there exists $p_0$ such that every random element $(f(x))_{x\in\ZZ^n}$ of class $C(n,\chi,p)$ with $p>p_0$ dominates stochastically an i.i.d.\ family of Bernoulli random variables $(g(x))_{x\in\ZZ^n}$ such that $\PP(g(x)=1)=\rho$.
Moreover, $\rho$ can be taken arbitrarily close to $1$ provided that $p_0$ is also made sufficiently close to $1$.
\end{theorem}

Let $\{\eta(x,y)\}_{(x,y)\in\ZZ^2}$ be such that
\begin{equation}
\label{e:etaxy}
\eta(x,y)=\begin{cases}
1, \mbox{ if all sites in }\Gamma(x,y) \mbox{ are open, }\\
0, \mbox{ otherwise.}
\end{cases}
\end{equation}

\begin{lemma}
\label{lemma the subgraph G} 
There exists $\chi\in\NN$ and $s=s(\EE_{\mathbf{p}}(\omega(o)))$ such that, under $\mathbb{P}_\mathbf{p}$, the process $\eta$ given by \eqref{e:etaxy} is of class $C(2,\chi,s)$. 
Furthermore, $s$ can be made arbitrarily close to $1$ provided that all the parameters $p_I$ is made close enough to $1$.
\end{lemma}
\begin{proof} For $x\in\ZZ^n$ and $R>0$, let $B(x;R)$ be the set of sites $y\in\ZZ^n$ such that $\|x-y\|\leq R$.
Let $R=\|w_1\| = \|w_2\|$. 
Observe that for each $v\in\ZZ^2$ we have that $\Gamma(v)\subset B(Av;R)$. In particular, for all $I\in\mathcal{I}(k;n)$, $\pi_I( \Gamma(v) )\subset \pi_I( B (\pi_I\circ A v; R ) )$. 
By Lemma \ref{lemma construcao da base da rede}-$ii$.\ there exists $c>0$ such that
$$\|\pi_I(Av-Au)\|\geq c\|u-v\|.$$
Thus, the choice $\chi=3R/c$ gives that $\pi_I(B(\pi_I \circ Au; R ) )$ and $\pi_I( B (\pi_I\circ A v; R ) )$  are disjoint provided that $\|v-u\|\geq \chi$. In particular $\eta(v)$ is independent of $\{\eta(u):u\in\ZZ^2,\|v-u\|\geq\chi\}$.

Since $\Gamma(v) \subset B(Av;R)$ we can use \eqref{e:box_open} to obtain
\begin{align*}
\PP_\mathbf{p}(\eta(v)=1)
\geq \PP_\mathbf{p}\big(\mbox{all sites in }B(Av,R)\mbox { are open} \big)
\geq \big[\EE_\mathbf{p}(\omega(o))\big]^{(2R+1)^k}.
\end{align*}
Hence $\eta$ is of class $C\big(2,3R/c,[\EE_\mathbf{p} (\omega(o))]^{(2R+1)^k}\big)$. 
\qed
\end{proof}

We are now ready to present the proof of \ref{theorem existence of percolation}.
\begin{proof}
[Proof of Theorem \ref{theorem existence of percolation}] 
In light of Theorem \ref{remark Ligget}, we can  choose all $p_I<1$ sufficiently close to $1$ so that $\EE \omega(o)=\prod_{I\in\mathcal{I}}p_I$  is large enough to guarantee that the process $\eta$ defined in \eqref{e:etaxy}  dominates stochastically a standard supercritical Bernoulli site percolation process in $\ZZ^2$.
In particular, with positive probability we have that $\mathcal{V}_\eta(o)$ is infinite. We conclude by observing that each path $\{o,v_1,v_2\ldots,\}\in\ZZ^2$ of $\eta$-open sites such that $\lim_{j\to\infty}\|v_j\|=\infty$ can be mapped into a path  $\{o,x_1,x_2,\ldots\}\subset\mathcal{G}$ with $\lim_{j\to\infty}\|x_j\|=\infty$ and whose sites are $\omega$-open. 
\qed
\end{proof}

\section{Polynomial decay of connectivity}
\label{sec:polynomial_decay}

In this section we prove Theorem \ref{t:poly_dec_1} and indicate the few modifications that lead to the proof of Theorem \ref{theorem power law projetando em k}.
Our method follows essentially the ideas presented in \cite{HilSid} for Bernoulli line percolation.
However, there is a complication and we need to adapt Lemma 4.7 therein to the higher dimension setting. 
The main problem is that the proof presented in \cite{HilSid} only works in $3$-dimensions.
We replace that result by our Proposition \ref{lemma dos caminhos} whose proof relies on Lemma \ref{lemma dos random walks}.

\subsection{Crossing events}
Given integers $a<b$ and $c<d$ and $\{\eta(x)\}_{x\in\ZZ^2} \in \{0,1\}^{\mathbb{Z}^2}$ we say that there is a bottom to top crossing in the rectangle $R:=[a,b]\times[c,d]\cap\ZZ^2$ if there is a path $\{(x_0,y_0),\ldots,(x_T,y_T)\}\subset R$ of
$\eta$-open sites such that $y_0=c$ and $y_T=d$.
We denote $\mathcal{BT}(R)$ the event that such a cross occur that is, the set of all the configurations $\eta$ for which there is a bottom to top crossing in $R$.
Similarly, we say that there is a left to right crossing in the rectangle $R$ if there exists a path of open sites $\{(w_0,z_0),\ldots,(w_H,z_H)\}\subset R$ with $w_0=a$ and $w_H=b$, and similarly, we denote this event by $\mathcal{LR}(R)$.

Let $N$ and $\{m_j\}_{j=2}^n$ be non-negative integers and denote
\[
B=B(N,m_2,\ldots,m_n):=[0,N]\times[0,m_2]\times\ldots\times[0,m_n]\cap\ZZ^n.
\]
Throughout this section we will regard the first coordinate as measuring the height of the rectangle $B$. 
Thus, for a random element  $\{\eta(x)\}_{x\in\ZZ^n} \in \{0,1\}^{\mathbb{Z}^n}$, we can refer to bottom to top crossings in $B$: We denote by $\mathcal{BT}(B)$ the set of all the configurations for which there exists a path of $\eta$-open sites $\{x_0,\ldots,x_T\}\subset B$ such that $\pi_{\{1\}}(x_0)=0$ and $\pi_{\{1\}}(x_T)=N$.

Let $k=2$ and $I_j=\{1,j\}$, $2\leq j \leq n$.
Notice that the set $\pi_{I_j}(B)\subset \ZZ_{I_j}^2$ is isomorphic to a rectangle in $\ZZ^2$ with side lengths $m_j$ and $N$ corresponding to the $j_{th}$ and first coordinate respectively.
Define
\[
\xi(x) = \omega_{I_2}(\pi_{I_2}(x)) \cdots \omega_{I_n}(\pi_{I_n}(x))
\]
Notice that $\xi \leq \omega$ (see \eqref{eq:def_omega}).
If there is a bottom to top crossing in $B$ of sites $x \in \mathbb{Z}^n$ that are $\omega_{I_j}$-open for all $2\leq j\leq n$ (which is to say, $\xi \in \mathcal{BT}(B)$) then a simple projection onto the coordinate planes $\mathbb{Z}^2_{I_j}$ show that $\omega_{I_j} \in \mathcal{BT}(\pi_{I_j}(B))$ for all $2\leq j\leq n$.
Our next result states that the converse is also true.
For that, given paths that cross the projections $\pi_{I_j}(B)$ from top to bottom we will need to construct a path inside $B$ which is projected under $\pi_{I_j}$ to the given crossing in $\pi_{I_j}(B)$.
Although it may sound somewhat intuitive that it is possible to do so, we did not find any existing proof for this fact. 
Therefore, we have produced a combinatorial proof that may be interesting in its own.

\begin{proposition}
\label{lemma dos caminhos} 
Let $k=2$.
For each $i=2,\ldots,n$, let $\gamma_{I_j}:\{0\} \cup [H_j]\to \pi_{I_j}(B)$ be a path composed of $\omega_j$-open sites such that: $\pi_{\{1\}}\circ\gamma_{I_j}(0)=0$, $\pi_{\{1\}}\circ\gamma_{I_j}(H_j)=N$ and for all $0\leq t <H_j$ we have $\pi_{\{1\}}\circ\gamma_{I_j}(t)<N$.
Then there exists a path $\lambda\colon \{0\}\cup [T] \to B$ whose sites are $\omega_{I_j}$-open and satisfying that for every $j =2, \ldots, n$, $\pi_{I_j} (\lambda(0)) = \gamma_{I_j}(0)$ and that $\pi_{I_j} (\lambda(T)) = \gamma_{I_j} (H_j)$.
In particular, if $\omega_{I_j}\in \mathcal{BT}(\pi_{I_j}(B))$ for all $2\leq j \leq n$, then there exists a bottom to top crossing in $B$ whose sites are $\omega_{I_j}$-open for each $2\leq j\leq n$.

\end{proposition}
Proposition \ref{lemma dos caminhos} implies that $[\xi \in \mathcal{BT}(B)]=\bigcap_{j=2}^n\mathcal{BT}(\pi_{I_j}(B))$ hence, by independence:
\begin{equation}\label{equation probabilidade path lemma}
\PP(\xi \in \mathcal{BT}(B))=\prod_{j=2}^n \PP_{p_{I_j}}( \mathcal{BT}(\pi_{I_j}(B)),
\end{equation}
where $\mathbb{P}$ stands for $ \otimes_{I \in \mathcal{I}(k;n)} \mathbb{P}_{p_I}$.
For $n=3$ and $k=2$, this result has already been proved in \cite[Lemma 4.7]{HilSid}. 
Here we extend this result for any $n\geq 3$ and $k=2$.

For the proof of Proposition \ref{lemma dos caminhos} we use the following lemma that is inspired by the Two Cautious Hikers Algorithm discussed by G.~Pete in \cite[page 1722]{G.Pete}.

\begin{lemma}
\label{lemma dos random walks} 
Let $N$ and $\{T_i\}_{i=1}^n$ be non-negative integers.
Let, for each $i\in[n]$,  $S_i:\{0\}\cup [T_i]\to \{0\}\cup  [N]$ be functions satisfying:\\
\noindent i. $|S_i(t)-S_i(t-1)|=1$, $\forall t\in [T_i]$ ;\\
ii. $0\leq S_i(t)< N$, for all $0\leq t < T_i$;\\
iii. $S_i(0)=0$ and $S_i(T_i)=N$.

\noindent Then there exists $T\in\NN$ and $f_i:\{0\}\cup [T]\to\{0\}\cup [T_i]$, $1\leq i\leq n$, that satisfy:\\
a) $|f_i(t)-f_i(t-1)|=1$, for all $t\in[T]$;\\
b) $S_1\circ f_1(t) = S_j \circ f_j(t)$, for all $t\in\{0\}\cup [T]$, for each $1\leq j\leq n$;\\
c) $S_1(f_1(0))=0$ and $S_1(f_1(T))=N$.
\end{lemma}

Before we give a proof for this lemma let us clarify its statement.
The functions $S_i$ can be thought of as $n$ different random walks parametrized by $t \in \{0, \dots T_i\}$ and that can at each step jump one unit up, one unit down or remain put.
They are required to start at height $0$, to remain above $0$ and finish at height $N$.
The conclusion is that it is possible to introduce delays to the individual random walks or even require them to backtrack (by means of composing them the $f_i$'s) so that they will all be parametrized by the same interval $\{0,\ldots, T\}$ and always share the same height for any time inside this interval.
The arguments in \cite{G.Pete} can be modified in order to obtain a proof for $n=2$.
Below we present a proof that works for general $n$.

\begin{proof}[Proof of Lemma \ref{lemma dos random walks}] 
Let $ G=(V(G), E(G))$ be a graph with vertex set $V(G) = \big\{v=(t_1,\ldots,t_n)\in\NN^n \colon\, \text{$t_i \in \{0,\ldots, T_i\}$ and $S_1(t_1)=S_j(t_j)$ for every $1\leq j\leq n$}\big\}$ and whose edge set $E(G)$ consists of the pairs of vertices $v=(t_1,\ldots,t_n)$ and $w=(s_1,\ldots,s_n)$ such that $|t_i-s_i|=1$, for all $1\leq i \leq n$. 
Similarly to \cite{G.Pete}, we have:

\noindent\textbf{Claim:} \textit{The degrees of every vertex $v\in V(G)$ are even, except for $(0,\ldots,0)$ and $(T_1,\ldots,T_n)$ that are the unique vertices that have degree $1$}.

\noindent\textit{Proof of the claim:}
It is simple to verify that $(0,\ldots,0)$ and $(T_1,\ldots,T_n)$ have degree one.

For $i \in [n]$ and $t \in \{0\} \cup [T_i]$, we say that $t$ is of type:
\begin{align*}
&(i, \slash)\,\,\, \mbox{ if } \,\,\, S_i(t+1)=S_i(t-1)+2;\\
&(i, \backslash)\,\,\, \mbox{ if } \,\,\, S_i(t+1)=S_i(t-1)-2;\\
&(i, \lor)\,\,\, \mbox{ if } S_i(t+1)=S_i(t-1) \mbox{ and } \,\,\, S_i(t+1)=S_i(t)+1;\\
&(i, \land)\,\,\, \mbox{ if } S_i(t+1)=S_i(t-1) \mbox{ and } \,\,\, S_i(t+1)=S_i(t)-1.
\end{align*}

Notice that if $v=(t_1,\ldots,t_n)\in V(G)$ has at least one $t_i$ of type $(i, \lor)$ and at least one $t_j$ of type $(j,2,-)$ its degree in G has to be equal to zero. 
This is because
$S_i(t_i\pm 1)=S_i(t_i)+1=S_j(t_j)+1$, and $S_j(t_j\pm 1)=S_j(t_j)-1$, and hence $S_i(t_i\pm 1)-S_j(t_j\pm 1)=2$ which implies that every possibility for the entries $(t_i\pm1)$ and $(t_i\pm1)$ of a neighbor of $v$ would lead to an element that does not belong to $V(G)$. 
Thus, a necessary condition for the degree of $v$ to be different from $0$ is that there is a partition $[n]=A\cup B$ such that for all $j\in B$, $t_j$ is of type $(j, \slash)$ or $(j, \backslash)$ and, for all $i\in A$ either every $t_i$ is of type $(i, \lor)$ or every $t_i$ is of type $(i, \land)$.

In the case that $v=(t_1,\ldots,t_n)\in G$ is such that all $t_i$ are of type $(i, \slash)$ or $(i, \backslash)$ we have that $v$ has exactly two neighborhoods $w$ and $w' \in V(G)$:
\begin{align*}
w=\big(t_i+\ind_{\{\text{$t_i$ is of type $(i, \slash)$}\}} - \ind_{\{\text{$t_i$ is of type $(i, \backslash)$}\}}\big)_{i=1}^n,\\
w'=\big(t_i-\ind_{\{\text{$t_i$ is of type $(i, \slash)$}\}} + \ind_{\{\text{$t_i$ is of type $(i, \backslash)$}\}}\big)_{i=1}^n.
\end{align*}

In the case that $v=(t_1,..,t_n)\in V(G)$ is such that exactly $k$ entries, say $\{t_{i_j}\}_{j=1}^k$ are of type $(i, \lor)$ and the other $n-k$ entries are of type $(i, \slash)$ or $(i, \backslash)$, we obtain that
$v$ has exactly $2^k$ neighborhoods. 
This follows by induction on $k$, by observing that each of the two possibilities $t_{i_j}\pm1$ imply $S_{i_j}(t_{i_j}\pm 1)=S_{i_j}(t_{i_j})+1$. 
Similarly, this is also true in the case that exactly $k$ distinct $\{t_{i_j}\}_{j=1}^k$ are of type $(i, \land)$ and the other $n-k$ are of type $(i, \slash)$ or $(i, \backslash)$. 
This completes the proof of the claim.

Let $v^*=(T_1,\ldots,T_n)$ and $o=(0,\ldots,0)$ be the unique vertices of $G$ that have degree $1$. 
Let $\mathcal{H}$ be the largest connected subgraph of $G$ that contains $o$. The sum $\sum_{v\in\mathcal{H}}\text{degree($v$)}$ is twice the number edges of $\mathcal{H}$, in particular it is an even number. 
Thus, $\sum_{v\in\mathcal{H}\setminus \{o\} }\text{degree($v$)}$ is an odd integer, and this holds if and only if $v^*\in\mathcal{H}$. 
Hence, $o$ and $v^*$ are in the same connected component of $G$, which implies the existence of a number $T\in\NN$ and path $\gamma:\{0\} \cup [T]\to G$ with $\gamma(0)=o$ and $\gamma(T)=v^*$. 
The choice $f_i(t)=\pi_{\{i\}}(\gamma (t))$ completes the proof.
\qed
\end{proof}
\begin{proof}[Proof of Proposition \ref{lemma dos caminhos}] 
Assume that for each $2\leq j\leq n$, there exist non-negative integers $H_j$ and paths of $\omega_{I_j}$-open sites $\gamma_{I_j}:\{0\} \cup [H_j]\to \pi_{I_j}(B)$ such that: $\pi_{\{1\}}\circ\gamma_{I_j}(0)=0$, $\pi_{\{1\}}\circ\gamma_{I_j}(H_j)=N$ and for all $0\leq t <H_j$ we have $\pi_{\{1\}}\circ\gamma_{I_j}(t)<N$.

For each $2\leq j\leq n$ we define recursively a sequence of times $(\tau_j(0),\tau_j(1),\ldots,\tau_j(T_j))$ as follows:
\begin{align*}
\tau_j(0)=&\, 0;\\
\tau_j(t+1)=&\, \inf\{s\in[\tau_j(t)+1,H_j]\cap\NN:|\pi_{\{1\}}(\gamma_{I_j}(s)-\gamma_{I_j}(s-1))|=1\},
\end{align*}
and we define $T_j$ as the first time at which $\tau_j(T_j)=H_j$.

Let $S_j:=\pi_{\{1\}}\circ \gamma_{I_j} \circ \tau_j$, for each $2\leq j \leq n$. 
Then $\{S_j\}_{j=2}^n$ satisfy Conditions \emph{i.-iii.}\ in Lemma \ref{lemma dos random walks}. 
Hence, there exists a non negative integer $T>0$ and $\{f_j\}_{j=2}^n$ that satisfy Conditions \emph{a-c} in the same lemma. 
We claim that the function $\lambda:\{0\}\cup [T]\to\ B$ denoted by $\lambda(t)=(\lambda_1(t),\ldots,\lambda_n(t))$ where
\begin{align*}
\lambda_1(t)&= S_2\circ f_2(t) = \cdots = S_n\circ f_n(t),\\
\lambda_j(t)&= \pi_{\{j\}} \circ \gamma_{I_j} \circ \tau_j \circ f_j(t),\;\;2\leq j\leq n,
\end{align*}
has the following properties:\\
\emph{i.} $\lambda(t)$ is $\omega_{I_j}$-open for all $2\leq j\leq n$;\\
\emph{ii.} For each $1 \leq t\leq T$, $\lambda(t)$ and $\lambda(t-1)$ are connected by a path $\gamma_t$ of sites that are $\omega_{I_j}$-open for all $2\leq j\leq n$;\\
\emph{iii.} $\lambda_1(0)=0$ and $\lambda_1(T)=N$.

The proof will be complete once we show that these three conditions are valid.
One can check readily that Condition \emph{iii}.\ holds.
So we now prove the validity of the other two.

\noindent \textit{Validity of Condition i.} Notice that, for all $t\in[T]$,
\begin{align*}
\pi_{I_j}(\lambda(t) )=\, &\big(\lambda_1(t),\lambda_j(t)\big)
                      =\, \big(S_j\circ f_j(t),\pi_{\{j\} }\circ\gamma_{I_j}\circ \tau_j\circ f_j(t)  \big)\\
                      =\, & \big(\pi_{\{1\}}\circ \gamma_{I_j} \circ \tau_j,\pi_{\{j\} }\circ\gamma_{I_j}\circ \tau_j\circ f_j(t)  \big) =  \gamma_{I_j}\circ \tau_j\circ f_j(t).
\end{align*}
Since every site in $\gamma_{I_j}$ is $\omega_{I_j}$-open, we have $\omega_{I_j}(\pi_{I_j}(\lambda(t) ))= \omega_{I_j}(\gamma_{I_j}\circ \tau_j\circ f_j(t))=1$.

\noindent \textit{Validity of Condition ii.} For each $t\in[T]$ we have either $\lambda_1(t)=\lambda_1(t-1) + 1$ or $\lambda_1(t)=\lambda_1(t-1) - 1$.
Let us assume the former holds.
The latter can be treated similarly.

There are integers $\{x_j\}_{j=2}^n$ such that
\[
\lambda(t)-\lambda(t-1)=(1,x_2,x_3,\ldots,x_n),
\]
and in particular:
\[
\gamma_{I_j}\circ \tau_j \circ f_j(t)-\gamma_{I_j}\circ \tau_j \circ f_j(t-1)  = e_1+x_je_j.
\]
We claim that $\lambda(t)-e_1$ is $\omega_{I_j}$-open for all $2\leq j\leq n$. 
To see this, notice that $\pi_{I_j}(\lambda(t)-e_1)=\gamma_{I_j}\circ \tau_j \circ f_j(t) -e_1$, and that only two possibilities may happen: either $f_j(t)-f_j(t-1)=-1$ or $f_j(t)-f_j(t-1)=+1$. In any case, denoting $a_j=f_j(t)-f_j(t-1)$, we have
\[
\gamma_{I_j}\circ \tau_j \circ f_j(t) -e_1 = \gamma_{I_j}( \tau_j\circ f_j(t) -a_j ),
\]
and hence $\lambda(t)-e_1$ also is $\omega_{I_j}$-open for all $2\leq j\leq n$.
Let $p_0=\lambda(t)$, $p_1=\lambda(t)-e_1$, and $p_k=p_{k-1}-x_{k}e_{k}$, $2\leq k\leq n$. In particular, $p_{n}=\lambda(t-1)$. Let $[p_j,p_{j+1}]$ be the path that goes along the line segment of points $x\in\ZZ^n$ that has $p_j$ and $p_{j+1}$ as its extremes. 
We claim that $\gamma_t=\bigcup_{j=1}^n [p_{j-1},p_{j}]$ is a path of sites fulfilling Condition \emph{ii.} Start with $y$ in the line segment $[p_1,p_2]$. 
Thus, $\pi_{1}(y)=\pi_{1}(p_1)$ and it is $\omega_{I_j}$-open for all $2\leq k\leq n$: The case $k=2$, follows from the definition of $\gamma_{I_2}$ and $\tau_2$; for $k>2$ we have $\omega_{I_k}(\pi_{I_k}(y)) = \omega_{I_k}( \pi_{I_k}(p_1))=1$. Inductively, each $y$ in the line segment $[p_j,p_{j+1}]$ has the following properties:
\begin{align*}
\pi_{1}(y)= &\,\, \pi_{1}(p_1)\\
\omega_{I_k} ( \pi_{I_k}(y) )= &\,\,  \omega_{I_k} (\pi_{I_k}(p_n)) = 1, \mbox{ for each }2\leq k\leq j,\\
\omega_{I_{j+1}} ( \pi_{I_{j+1}}(y) )= &\,\, 1, \mbox{ by the definition of } \gamma_{I_{j+1}} \mbox{ and }\tau_{j+1},\\
\omega_{I_k} ( \pi_{I_k}(y) )= &\,\, \omega_{I_k} (\pi_{I_k}(p_1)) = 1, \mbox{ for each }j+1<k\leq n.\\
\end{align*}
This completes the proof. 
\qed
\end{proof}

\subsection{Percolation on a renormalized lattice}
In this section we will define a percolation process in a renormalized lattice whose sites can be matched to hypercubes from the original lattice, called boxes.
We restrict ourselves to the case $k=2$ so that the projections of theses boxes into the coordinate planes are given by squares.
The boxes are called good depending on whether some crossings occur inside and around some of these projected squares as illustrated in Figure \ref{fig:good_boxes}.
Taking the side of the boxes to be large enough, we can guarantee that boxes are good with large probability, so that the percolation processes induced by good boxes in the renormalized lattice dominates a supercritical site percolation process.
This implies the existence of arbitrarily long paths of good boxes.
The directed nature of the dependencies introduces some complications in showing stochastic domination, and in order to gain some kind of independence we will need to look at oriented portions of the renormalize lattice as the one illustrated in Figure \ref{fig:wall_boxes}.

Before we proceed, we recall a classical fact about crossing events.
\begin{remark}
\label{remark correlation length} 
Let $p>p_c(\ZZ^2)$.
If $c>0$ is large enough (depending on $p$) then
\begin{equation}
\label{e:correlation_length}
\lim_{N\to\infty}\PP_p( \mathcal{BT}([0,\lfloor c\log N \rfloor]\times[0,N]\cap \mathbb{Z}^2))=1,
\end{equation}
where and $\lfloor x\rfloor$ denotes the integer part of a real number $x$.
\end{remark}

From now on, given $y\in\ZZ^n$, we use the notation
\[
B(y;N)=:[y_1,y_1+N-1]\times\ldots\times[y_n,y_n+N-1]\cap\ZZ^n.
\] 
Furthermore, we define:
\begin{align*}
B_j(y;N):=\pi_{I_j}\big{(}B(y;N) \cup B(y+Ne_1;N) \cup B(y+Ne_j;N)\big{)}.
\end{align*}
In the next definition, we still stick with the convention that the first coordinate $x_1$ measures the height of each $B_j(y;N)$:
\begin{definition}
\label{definicao caixa boa}
Given $\omega_{I_j} \in \{0,1\}^{\mathbb{Z}^k_{I_j}}$, we say $B_j(y;N)$ is $\omega_{I_j}$-good when 
\[
\omega_{I_j} \in \mathcal{BT}(\pi_{I_j} (B(y;N)\cup B(y+Ne_1;N))) \cap \mathcal{LR}(\pi_{I_j} (B(y;N)\cup B(y+Ne_j;N))).
\]
Furthermore, given $(\omega_{I_j})_{j=2}^n$, we say that $B(y;N)$ is good if for all $j=2,\ldots,n$, $B_j(y;N)$ is $\omega_{I_j}$-good (see Figure \ref{fig:good_boxes}).
\end{definition}

\begin{figure}[htb!]
\centering
		\begin{tikzpicture}[scale=.4]
	
		\foreach \x in {1,...,2}{
				 \draw[thin, fill=black!15!] (0,-2*\x-.60*\x)--(1,-2*\x-.60*\x+.60)--(2,-2*\x-.60*\x)--(1,-2*\x-.60*\x-.60)--(0,-2*\x-.60*\x);

				 \draw[thin, fill=black!40!] (0,-2*\x-.60*\x)--(0,-2*\x-.60*\x-1.40)--(1,-2*\x-.60*\x-2)--(1,-2*\x-.60*\x-.60)--(0,-2*\x-.60*\x);
				  \draw[thin, fill=black!15!] (1,-2*\x-.60*\x-.60) -- ++ (0,-1.4) -- ++ (1,.6) -- ++ (0,1.4) -- ++ (-1,-.6);
				 }
				 
				  \foreach \x in {2,...,2}{
				 \draw[thin, fill=black!15!] (-1,-2*\x-.60*\x+.60)--(0,-2*\x-.60*\x+.60+.60)--(1,-2*\x-.60*\x+.60)--(0,-2*\x-.60*\x-.60+.60)--(-1,-2*\x-.60*\x+.60);
				 \draw[thin, fill=black!40!] (-1,-2*\x-.60*\x+.60)--(-1,-2*\x-.60*\x-1.40+.60)--(0,-2*\x-.60*\x-2+.60)--(0,-2*\x-.60*\x-.60+.60)--(-1,-2*\x-.60*\x+.60);
				 }
				 
				 \foreach \x in {1,...,1}{
				 \draw[thin, fill=black!15!] (-1,-2*\x-.60*\x-.60)--(0,-2*\x-.60*\x)--(1,-2*\x-.60*\x-.60)--(0,-2*\x-.60*\x-1.20)--(-1,-2*\x-.60*\x-.60);

				 \draw[thin, fill=black!40!] (-1,-2*\x-.60*\x-.60)--(-1,-2*\x-.60*\x-2)--(0,-2*\x-.60*\x-2.60)--(0,-2*\x-.60*\x-1.20)--(-1,-2*\x-.60*\x-.60);
				 
				  \draw[thin, fill=black!15!] (0,-2*\x-.60*\x-1.2) -- ++ (0,-1.4) -- ++ (1,.6) -- ++ (0,1.4) -- ++ (-1,-.6);
				 }
				 
				 \foreach \x in {3,...,3}{
				 \draw[thin] (1-\x, -15+2*\x) -- ++ (-1,.6) -- ++ (0,1.4) -- ++(1, -.6) -- ++(0,-1.4);
				  \draw[thick] (1-\x-.25, -15+2*\x+0.15) to [out=90, in=300] ++ (-.25, 0.15+1.4) to [out=120, in=270] ++ (-.25, 0.15+1.4);
				  \draw[thick] (1-\x, -15+2*\x+1.05) to [out=140, in=30] ++ (-1,-.35+.6)  to [out=210, in=320] ++ (-1,-.35+.6);
				 \draw[thin] (-\x, -14.4+2*\x) -- ++ (-1,.6) -- ++ (0,1.4) -- ++(1, -.6) -- ++(0,-1.4);
				 \draw[thin] (1-\x, -15+2*\x+1.4) -- ++ (-1,.6) -- ++ (0,1.4) -- ++(1, -.6) -- ++(0,-1.4);
				 }
				 
				 \foreach \x in {3,...,3}{
				  \draw[thick, shift={(-1 cm,.6 cm)}] (1-\x-.25, -15+2*\x+0.15) to [out=90, in=300] ++ (-.25, 0.15+1.4) to [out=120, in=270] ++ (-.25, 0.15+1.4);
				  \draw[thick, shift={(-1 cm,.6 cm)}] (1-\x, -15+2*\x+1.05) to [out=140, in=30] ++ (-1,-.35+.6)  to [out=210, in=320] ++ (-1,-.35+.6);
				 \draw[thin, shift={(-1 cm,.6 cm)}] (1-\x, -15+2*\x) -- ++ (-1,.6) -- ++ (0,1.4) -- ++(1, -.6) -- ++(0,-1.4);	 
				 \draw[thin, shift={(-1 cm,.6 cm)}] (-\x, -14.4+2*\x) -- ++ (-1,.6) -- ++ (0,1.4) -- ++(1, -.6) -- ++(0,-1.4);
				 \draw[thin, shift={(-1 cm,.6 cm)}] (1-\x, -15+2*\x+1.4) -- ++ (-1,.6) -- ++ (0,1.4) -- ++(1, -.6) -- ++(0,-1.4);
				 }
				 
				 \foreach \x in {3,...,3}{
				 \draw[thick, shift={(-1 cm, 2 cm)}] (1-\x-.25, -15+2*\x+0.15) to [out=90, in=300] ++ (-.25, 0.15+1.4) to [out=120, in=270] ++ (-.25, 0.15+1.4);
				  \draw[thick, shift={(-1 cm, 2 cm)}] (1-\x, -15+2*\x+1.05) to [out=140, in=30] ++ (-1,-.35+.6)  to [out=210, in=320] ++ (-1,-.35+.6);
				  \draw[thin, shift={(-1 cm,2 cm)}] (1-\x, -15+2*\x) -- ++ (-1,.6) -- ++ (0,1.4) -- ++(1, -.6) -- ++(0,-1.4);			 
				 \draw[thin, shift={(-1 cm, 2 cm)}] (-\x, -14.4+2*\x) -- ++ (-1,.6) -- ++ (0,1.4) -- ++(1, -.6) -- ++(0,-1.4);
				 \draw[thin, shift={(-1 cm,2 cm)}] (1-\x, -15+2*\x+1.4) -- ++ (-1,.6) -- ++ (0,1.4) -- ++(1, -.6) -- ++(0,-1.4);
				 }

				  \foreach \x in {3,...,3}{
				 \draw[thin] (1+\x, -15+2*\x) -- ++ (1,.6) -- ++ (0,1.4) -- ++(-1, -.6) -- ++(0,-1.4);
				 \draw[thick] (1+\x + .25, -15+2*\x+.15) to [out=90, in=240] ++ (.25, 0.15+1.4) to [out=60, in=270] ++ (.25, 0.15+1.4);
				 \draw[thick] (1+\x, -15+2*\x+1.05) to [out=40, in=150] ++ (1,-.35+.6)  to [out=330, in=220] ++ (1,-.35+.6);
				 \draw[thin] (2+\x, -14.4+2*\x) -- ++ (1,.6) -- ++ (0,1.4) -- ++(-1, -.6) -- ++(0,-1.4);
				 \draw[thin] (1+\x, -15+2*\x+1.4) -- ++ (1,.6) -- ++ (0,1.4) -- ++(-1, -.6) -- ++(0,-1.4);
				 }
				 
				 \foreach \x in {3,...,3}{
				  \draw[thick, shift={(0 cm, 1.4 cm)}] (1+\x + .25, -15+2*\x+.15) to [out=90, in=240] ++ (.25, 0.15+1.4) to [out=60, in=270] ++ (.25, 0.15+1.4);
				 \draw[thick, shift={(0 cm, 1.4 cm)}] (1+\x, -15+2*\x+1.05) to [out=40, in=150] ++ (1,-.35+.6)  to [out=330, in=220] ++ (1,-.35+.6);
				 \draw[thin, shift={(0 cm, 1.4 cm)}] (1+\x, -15+2*\x) -- ++ (1,.6) -- ++ (0,1.4) -- ++(-1, -.6) -- ++(0,-1.4);
				 \draw[thin, shift={(0 cm, 1.4 cm)}] (2+\x, -14.4+2*\x) -- ++ (1,.6) -- ++ (0,1.4) -- ++(-1, -.6) -- ++(0,-1.4);
				 \draw[thin, shift={(0 cm, 1.4 cm)}] (1+\x, -15+2*\x+1.4) -- ++ (1,.6) -- ++ (0,1.4) -- ++(-1, -.6) -- ++(0,-1.4);
				 }
				 
				 \foreach \x in {3,...,3}{
				  \draw[thick, shift={(1 cm, 2 cm)}] (1+\x + .25, -15+2*\x+.15) to [out=90, in=240] ++ (.25, 0.15+1.4) to [out=60, in=270] ++ (.25, 0.15+1.4);
				 \draw[thick, shift={(1 cm, 2 cm)}] (1+\x, -15+2*\x+1.05) to [out=40, in=150] ++ (1,-.35+.6)  to [out=330, in=220] ++ (1,-.35+.6);

				 \draw[thin, shift={(1 cm, 2 cm)}] (1+\x, -15+2*\x) -- ++ (1,.6) -- ++ (0,1.4) -- ++(-1, -.6) -- ++(0,-1.4); 
				 \draw[thin, shift={(1 cm, 2 cm)}] (2+\x, -14.4+2*\x) -- ++ (1,.6) -- ++ (0,1.4) -- ++(-1, -.6) -- ++(0,-1.4);
				 \draw[thin, shift={(1 cm, 2 cm)}] (1+\x, -15+2*\x+1.4) -- ++ (1,.6) -- ++ (0,1.4) -- ++(-1, -.6) -- ++(0,-1.4);
				 }

\draw[->] (1,-15) -- ++(-10,6);
\node[below] at (-9,-9) {$e_2$};
\draw[->] (1,-15) -- ++(10,6);
\node[below] at (11,-9) {$e_3$};
\draw[->] (12,-7.8)--(12,-2.8);
\node[right] at (12,-2.8) {$e_1$};

\draw[dotted] (1,-7.2) -- ++ (-3,-1.8);
\draw[dotted] (0,-6.6) -- ++ (-3,-1.8);
\draw[dotted] (-1,-6.0) -- ++ (-3,-1.8);
\draw[dotted] (0,-3.8) -- ++ (-3,-1.8);
\draw[dotted] (-1,-3.2) -- ++ (-3,-1.8);

\draw[dotted] (1,-7.2) -- ++ (3,-1.8);
\draw[dotted] (2,-6.6) -- ++ (3,-1.8);
\draw[dotted] (0,-3.8) -- ++ (4,-2.4);
\draw[dotted] (2,-2.6) -- ++ (4,-2.4);
\draw[dotted] (2,-4.0) -- ++ (4,-2.4);

\draw[dotted] (-2,-9) -- ++ (0,-4.2);
\draw[dotted] (-5,-7.2) -- ++ (0,-4.2);
\draw[dotted] (4,-9) -- ++ (0,-4.2);
\draw[dotted] (6,-7.8) -- ++ (0,-4.2);

	\end{tikzpicture}
	\caption{The projection into the subspace spanned by $e_1$, $e_2$ and $e_3$ of a set of four adjacent boxes that are $\omega_{I_2}$-good and $\omega_{I_3}$-good.}	
	\label{fig:good_boxes}
\end{figure}
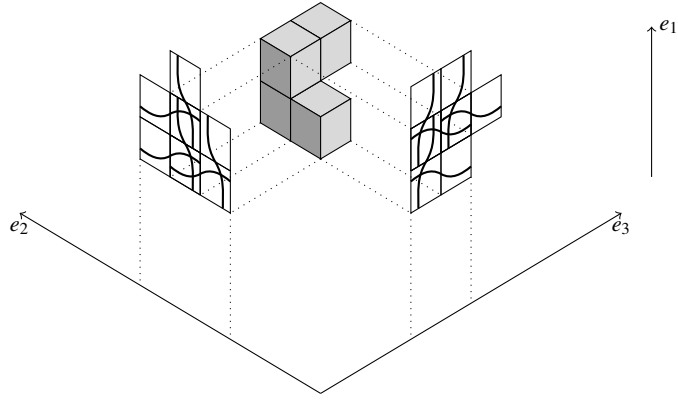

For fixed $N\in\NN$ we say that $\gamma:\NN \to\ZZ^n$ is a path of good boxes if
\begin{equation}
\label{eq:path_good_boxes}
\begin{split}
\text{for each $t\in\NN$, $B(\gamma(t);N)$ is a good box and}\\ 
\text{ $\gamma(t+1)=\gamma(t)+Ne_{j_t}$ for some $j_t\in[n]$.}
\end{split}
\end{equation}
In what follows we will make use of the following lemma whose proof can be done following exactly the same lines as in \cite[Lemmas 4.10 and 4.11]{HilSid}.
The idea is to iterate the use of Proposition \ref{lemma dos caminhos} to pass from one good box to the next one following a path contained inside these boxes whose sites are $\omega_{I_j}$-open for each $2\leq j \leq n$ (as done in \cite[Lemma 4.10]{HilSid}).
This is possible because the definition of good boxes entails the existence of a system of crossings inside the projections of these boxes into the respective $\mathbb{Z}^2_I$ for which Proposition  \ref{lemma dos caminhos} apply as shown in Figure \ref{fig:good_boxes}.
Being able to pass from one good box to the next adjacent one, all we need to do is to  concatenate the paths in order to obtain a path starting in the first good box in the sequence and ending at the last one (as done in \cite[Lemma 4.11]{HilSid}).

\begin{lemma}
\label{l:path_of_good} 
Let $N\in\NN$ and $\gamma:\NN\to\ZZ^n$ be a path of good boxes. 
Then for each $t\in\NN$, there exists a path of sites $\{x_0,\ldots,x_T\}\subset \cup_{s=0}^t B(\gamma(s);N)$ that are $\omega_{I_j}$-open for each $2\leq j \leq n$, and with $x_0\in B(\gamma(0);N)$ and $x_T\in B(\gamma(t);N)$. 
\end{lemma}

Let $p_0=o$ and $r_0=2$.
For $N\in\NN$ and $t\in\ZZ$, define recursively
\begin{equation}
\label{eq:pt}
p_t=p_{t-1}+N e_{r_t},
\end{equation}
 where $r_t\in\{2,3\ldots,n\}$ is such that $t\equiv r_t-2 \mod (n-1)$.
Roughly speaking, as $t$ increases the values of $r_t$ run through the set $\{2,3,\ldots, n\}$ cyclically. 
As an example, when $n=4$ we have $r_0=2, r_1=3, r_2=4, r_3=2, r_4=3, r_5=4,$ and so on.
As for the points $p_t$, they form a directed sequence whose increments are segments of length $N$, each oriented along one of the directions in $\{e_2,e_3,\ldots, e_n\}$.
The orientation of these segments follow the same cyclic pattern as $r_t$.

Let $\nu =(\nu(t,x))_{(t,x) \in\mathbb{Z}^2}$ be the random element in $\{0,1\}^{\ZZ^2}$ defined as
\begin{equation}
\label{e:def_nu}
\nu(t,x):=\ind [B(p_t+Nx e_1;N)\mbox{ is good}].
\end{equation}

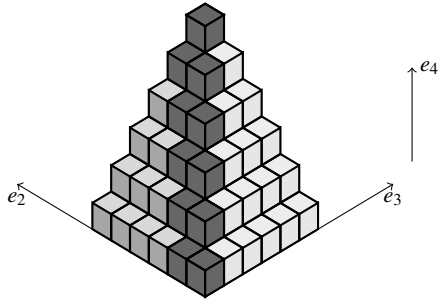
\begin{figure}[htb!]
	\centering
	\begin{tikzpicture}[scale=.25]
	
   				\foreach \x in {0,...,5}
    			{\foreach \i in {0,...,\x}
    			 {\pgfmathparse{.60*\i}
    			    \xdef\y{\pgfmathresult}
    	 		  \draw[fill=black!7!] (\x-\i+2,-2*\x-\y)--(\x-\i+1,-2*\x-\y+.60)--(\x-\i,-2*\x-\y)--(\x-\i+1,-2*\x-\y-.60)--(\x-\i+2,-2*\x-\y);
    				  \draw[fill=black!10!](\x-\i+2,-2*\x-\y)--(\x-\i+2,-2*\x-\y-1.40)--(\x-\i+1,-2*\x-\y-2)--(\x-\i+1,-2*\x-\y-.60);
    				  }}
		\foreach \x in {0,...,5}
			{\foreach \i in {0,...,\x}
			 {\pgfmathparse{.60*\i}
			    \xdef\y{\pgfmathresult}
				\draw[fill=black!15!](-\x+\i,-2*\x-\y)--(-\x+\i+1,-2*\x-\y+.60)--(-\x+\i+2,-2*\x-\y)--(-\x+\i+1,-2*\x-\y-.60)--(-\x+\i,-2*\x-\y);
				\draw[fill=black!35!] (-\x+\i,-2*\x-\y)--(-\x+\i,-2*\x-\y-1.40)--(-\x+\i+1,-2*\x-\y-2)--(-\x+\i+1,-2*\x-\y-.60)--(-\x+\i,-2*\x-\y);

				 \draw[fill=black!60!] (0,-2*\x-.60*\x)--(1,-2*\x-.60*\x+.60)--(2,-2*\x-.60*\x)--(1,-2*\x-.60*\x-.60)--(0,-2*\x-.60*\x);

				 \draw[fill=black!60!] (0,-2*\x-.60*\x)--(0,-2*\x-.60*\x-1.40)--(1,-2*\x-.60*\x-2)--(1,-2*\x-.60*\x-.60)--(0,-2*\x-.60*\x);
				 }}
				 
				 \foreach \x in {1,...,5}{
				 \draw[fill=black!60!] (-1,-2*\x-.60*\x+.60)--(0,-2*\x-.60*\x+.60+.60)--(1,-2*\x-.60*\x+.60)--(0,-2*\x-.60*\x-.60+.60)--(-1,-2*\x-.60*\x+.60);
				 \draw[fill=black!60!] (-1,-2*\x-.60*\x+.60)--(-1,-2*\x-.60*\x-1.40+.60)--(0,-2*\x-.60*\x-2+.60)--(0,-2*\x-.60*\x-.60+.60)--(-1,-2*\x-.60*\x+.60);
				 }
				 
		\foreach \x in {0,...,5}
			{\foreach \i in {0,...,\x}
			 {\pgfmathparse{.60*\i}
			    \xdef\y{\pgfmathresult}
			     \draw[thick] (-\x+\i,-2*\x-\y)--(-\x+\i+1,-2*\x-\y+.60)--(-\x+\i+2,-2*\x-\y)--(-\x+\i+1,-2*\x-\y-.60)--(-\x+\i,-2*\x-\y)--(-\x+\i,-2*\x-\y-1.40)--(-\x+\i+1,-2*\x-\y-2)--(-\x+\i+1,-2*\x-\y-.60);
			    	 \draw[fill=black!60!](\x-\x+2,-2*\x-.60*\x)--(\x-\x+2,-2*\x-.60*\x-1.40)--(\x-\x+1,-2*\x-.60*\x-2)--(\x-\x+1,-2*\x-.60*\x-.60);
				 	  
				 \draw[thick] (\x-\i+2,-2*\x-\y)--(\x-\i+1,-2*\x-\y+.60)--(\x-\i,-2*\x-\y)--(\x-\i+1,-2*\x-\y-.60)--(\x-\i+2,-2*\x-\y)--(\x-\i+2,-2*\x-\y-1.40)--(\x-\i+1,-2*\x-\y-2)--(\x-\i+1,-2*\x-\y-.60);
				 }}

\draw[->] (1,-15) -- ++(-10,6);
\node[below] at (-9,-9) {$e_2$};
\draw[->] (1,-15) -- ++(10,6);
\node[below] at (11,-9) {$e_3$};
\draw[->] (12,-7.8)--(12,-2.8);
\node[right] at (12,-2.8) {$e_4$};
	\end{tikzpicture}
	\caption{Let $n=4$ and $k=2$.  
	The darker boxes are the projections into the subspace spanned by $e_2, e_3, e_4$ of boxes $B(p_t+Nxe_1;N)$ for $t=0,\ldots, 15$ and arbitrary $x$ (light gray boxes where added to help visualization).
	 There are only $11$ such boxes that are visible from this perspective which correspond to indices $t$ with $r_t = 2$ or $r_t=4$.
	 The other $5$ boxes corresponding to $r_t=3$ are not visible because they lie behind other  boxes colored in light gray.}
	 \label{fig:spiral_boxes}
\end{figure}

\begin{figure}[htb!]
	\centering
	\begin{tikzpicture}[scale=.25]
		\foreach \x in {0,...,4}{
		\draw[fill=black!25!] (2*\x+1,.6) -- (2*\x+2,1.2) -- (2*\x+3,.6) -- (2*\x+2,0) -- (2*\x+1,.6);	}
		\foreach \x in {0,...,5}{
			\draw[fill=black!25!] (2*\x,0) -- (2*\x+1,.6) -- (2*\x+2,0) -- (2*\x+1,-.6) -- (2*\x,0);
			\foreach \y in {0,...,9}{
			\draw[fill=black!15!] (2*\x,-1.4*\y)--(2*\x,-1.4*\y-1.4)--(2*\x+1,-1.4*\y-2)--(2*\x+1,-1.4*\y-.6)--(2*\x,-1.4*\y);
			\draw[fill=black!40!] (2*\x+1,-1.4*\y-.6)--(2*\x+1,-1.4*\y-2)--(2*\x+2,-1.4*\y-1.4)--(2*\x+2,-1.4*\y);
			\draw[dotted] (11,-14.6) -- ++ (-5,-3);
			\draw[dotted] (12,-6.8) -- ++ (6,-3.6);
			\draw[dotted] (12,-14)--(18,-10.4);
			 \draw[->] (0,-14) -- ++(8,-4.8);
			 \node[below] at (8,-18.8) {$e_3$};
			 \draw[->] (12,-6.8) -- ++(2,1.2);
			 \node[above] at (14,-5.6) {$e_2$};
			 \draw[->] (0,0) -- ++(0,3);
			 \node[left] at (0,3) {$e_1$};
			}}
\end{tikzpicture}
\caption{Let $n=4$ and $k=2$. 
This picture shows the projection into the subspace spanned by $e_1, e_2$ and $e_3$ of the boxes $B(p_t + Nxe_1;N)$ for $t=0,\ldots,15$ and $x=0,\ldots,9$.}
\label{fig:wall_boxes}
\end{figure}
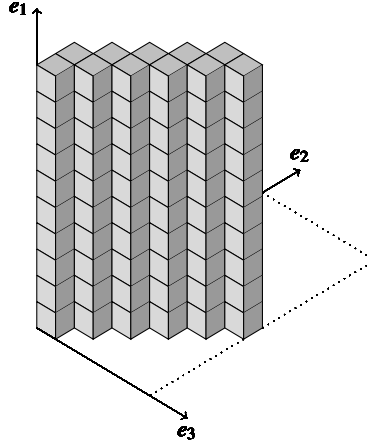

The process $\nu$ can be thought as a percolation process in a renormalized (that is, rescaled) lattice where sites are now boxes of the type $B(p_t+Nxe_1;N)$ (for $(t,x) \in \mathbb{Z}^2$).
The cyclic nature of $r_t$ and the choice of the sequence $p_t$ as in \eqref{eq:pt} allows us to derive some properties of this renormalized lattice. 
On the one hand, the projection into the subspace spanned by $e_2,\ldots e_n$ is given by a spiral  sequence of neighboring boxes which is isomorphic to a line of boxes sharing a face, see Figure \ref{fig:spiral_boxes}.
On the other hand, the projection into the subspace spanned by $e_1, e_i, e_j$ for  $i,j \in \{2,\ldots, n\}$ and $i \neq j$ resembles a jagged wall of boxes as illustrated in Figure \ref{fig:wall_boxes} which is isomorphic to a plane of adjacent boxes.
This specific shape guarantees that, for the process $\nu$, the statuses of distant boxes in this renormalized lattice are independent as they will have disjoint projections into the $\mathbb{Z}^2_{I_j}$ subspaces. 
This is the content of the next lemma:
\begin{lemma}
\label{lemma da cortina chi dependencia} 
For the random element $\nu$ given as in \eqref{e:def_nu}, the variable $\nu(t,x)$ is independent of  $\{\nu(s,y) \colon  |t-s|\geq 2(n-1) \text{ or } |y-x|\geq 2\}$.
In particular, the process $\nu$ is of class $C(2,\chi,p)$ for some $p>0$.
\end{lemma}
\begin{proof} Let $(t,x)\in\ZZ^2$ be fixed. 
Since the event that the box $B(y;N)$ is good is measurable with respect to the $\sigma$-algebra generated by the family of random variables
\[
\bigcup_{j=2}^n \{\omega_{I_j}(v):v\in B_j(y;N)\},
\]
all we need to show is that, for each $2\leq j \leq n$, the sets $B_j(p_t+Nxe_1;N)$ and $B_j(p_s+Nye_1;N)$ are disjoint in both cases $|t-s|\geq 2(n-1)$ or $|y-x|\geq 2$.

If $|y-x|\geq 2$ we have that for each $2\leq j\leq n$ the distance between $B_j(p_t+Nxe_1;N)$ and $B_j(p_s+Nye_1;N)$ is at least $1$, and this is also true in the case that $|t-s|\geq 2(n-1)$, since, in this case,  we have that except for the first coordinate, each coordinate of $p_s-p_t$ has absolute value at least $2N$. 
\qed
\end{proof}

Once the range of dependency is controlled for the process $\nu$, we use Theorem \ref{remark Ligget} in order to have it dominated from below by a supercritical Bernoulli process:

\begin{lemma}
\label{lemma dominacao estocastica} 
Assume that for each $2\leq j\leq n$ we have $p_{I_j}>p_c(\ZZ^2)$. 
Then for every $\epsilon>0$ there exists $N = N(\epsilon) \in\NN$ such that the process $\nu = (\nu(t,x))_{(t,x)\in \mathbb{Z}^2}$ (cf.\ \eqref{e:def_nu}) dominates stochastically a standard Bernoulli site percolation process in $\ZZ^2$ with parameter $p>1-\epsilon$.
\end{lemma}
\begin{proof} We write $\mathbb{P} = \otimes_{I \in \mathcal{I}(k;n)} \mathbb{P}_{p_I} $. 
Since for all $(t,x)\in\ZZ^2$ we have that $\PP(\nu(t,x)=1)=\PP(\nu(0,0)=1)$, in view of Theorem \ref{remark Ligget} and the previous Lemma, we only need to show that $\PP(\nu(0,0)=1)$ can be made arbitrarily close to $1$ for suitable large $N$.
Recall that the box $B(o;N)$ is good if for each $2\leq j\leq n$, $B_j(o;N)$ is $\omega_{I_j}$-good. 
Observe that for $j\neq l$, the events $[B_j(o;N) \text{ is $\omega_{I_j}$-good}]$ and $[B_l(o;N) \text{ is $\omega_{I_l}$-good}]$ are independent. 
Hence, all we need to show is that the assumption $p_{I_j}>p_c(\ZZ^2)$ implies that the probability that each $B_j(o;N)$ is $\omega_{I_j}$-good can be made arbitrarily close to $1$ for each $j$, for a suitable  choice of a large non negative integer $N$, which may depend on $p_j$. By the FKG inequality, we have:
\begin{align*}
\PP(B_j(o;N) \mbox{ is } \omega_{I_j}\mbox{-good }) \geq\,  & \PP_{p_{I_j}}\big(\mathcal{BT}(\pi_{I_j} (B(y;N)\cup B(y+Ne_1;N) ) \big) \times\\
                                                             &   \PP_{p_{I_j}} \big(\mathcal{LR}(\pi_{I_j} (B(y;N)\cup B(y+Ne_j;N)) \big).
\end{align*}
The fact that each probability in the right-hand side above can be made arbitrarily close to $1$ by choosing $N$ sufficiently large follows from the fact that $p_{I_j} > p_c(\mathbb{Z}^2)$ together with classical crossing probability estimates for supercritical Bernoulli site percolation (for instance Eq.\ \eqref{e:correlation_length} in Remark \ref{remark correlation length} is sufficient). 
\qed
\end{proof}

\subsection{Proof of Theorems \ref{t:poly_dec_1} and \ref{theorem power law projetando em k}.}
We start this section presenting the proof for Theorem \ref{t:poly_dec_1}.
Roughly speaking, it consists of three steps.
First find a path spanning a rectangle in the renormalized lattice that is very elongated in the vertical direction.
This path can be mapped to a path of good boxes in the original lattice.
Lemma \ref{l:path_of_good} allows to obtain a long path of sites in the original lattice that are $\omega_{I_j}$ open for every $j=2,\ldots,n$.
Comparison with a supercritical percolation process in the renormalized lattice, shows that this step can be accomplished paying only a constant probability cost.
The second step consist of guarantying that the sites in this long path are also $\omega_I$ open for every index $I$ that does not contain the coordinate $1$.
The geometry of our construction allows to accomplish this step by paying only a polynomial path in the length of the path.
An extra polynomial probability cost needs to be payed in order to require that the long path starts at the origin.
The third and last step consists in guarantying that the origin does not belong to an infinite connected component.
This accomplished by construction a closed set surrounding the origin much in the spirit of Lemma \ref{lemma nao perc} .
Again only polynomial probability cost is necessary to accomplish this step.

At the end of this section we indicate the modifications that need to be performed in the proof in order to obtain a proof of Theorem  \ref{theorem power law projetando em k}.

\begin{proof}[Proof of Theorem \ref{t:poly_dec_1}] 
Let $\epsilon>0$ be such that $1-\epsilon>p_c(\ZZ^2)$ and $\{\eta(x)\}_{x\in\ZZ^2}$ be a Bernoulli site percolation process on $\mathbb{Z}^2$ with parameter $1-\epsilon$, hence supercritical.
Fix $N\in\NN$ (depending on $\epsilon$) large enough so that the claim in Lemma \ref{lemma dominacao estocastica} holds, i.e., the process $\{\nu(x)\}_{x\in\ZZ^2}$ defined in \eqref{e:def_nu} dominates $\{\eta(x)\}_{x\in\ZZ^2}$ stochastically.
Let 
\[
\mathcal{O}_1=\big\{\nu\in \mathcal{BT}([0,\lfloor c_o\log K \rfloor] \times [0,K]\cap\ZZ^2)\big\}.
\] 
In view of Remark \ref{remark correlation length}, we can choose a large constant $c_o>0$ (depending on $\epsilon$) such that
\begin{equation}
\label{equation caminho de caixas boas}
\liminf_{K\to\infty} \PP\big(\mathcal{O}_1) \geq \liminf_{K\to\infty} \PP\big(\eta \in  \mathcal{BT}([0,\lfloor c_o\log K \rfloor] \times [0,K]\cap\ZZ^2)>0.
\end{equation}
The value of $c_o$ will be kept fixed from now on.

\begin{center}
\begin{figure}
\centering
	\begin{tikzpicture}[scale=.25]
		\foreach \x in {0,...,4}{
		\draw[fill=black!25!] (2*\x+1,.6) -- (2*\x+2,1.2) -- (2*\x+3,.6) -- (2*\x+2,0) -- (2*\x+1,.6);	}
		\foreach \x in {0,...,5}{
			\draw[fill=black!25!] (2*\x,0) -- (2*\x+1,.6) -- (2*\x+2,0) -- (2*\x+1,-.6) -- (2*\x,0);
			\foreach \y in {0,...,9}{
			\draw[fill=black!15!] (2*\x,-1.4*\y)--(2*\x,-1.4*\y-1.4)--(2*\x+1,-1.4*\y-2)--(2*\x+1,-1.4*\y-.6)--(2*\x,-1.4*\y);
			\draw[fill=black!40!] (2*\x+1,-1.4*\y-.6)--(2*\x+1,-1.4*\y-2)--(2*\x+2,-1.4*\y-1.4)--(2*\x+2,-1.4*\y);
			\draw[dotted] (11,-14.6) -- ++ (-5,-3);
			\draw[dotted] (12,-6.8) -- ++ (6,-3.6);
			\draw[dotted] (12,-14)--(18,-10.4);
			 \draw[->] (0,-14) -- ++(8,-4.8);
			 \node[below] at (8,-18.8) {$e_3$};
			 \draw[->] (12,-6.8) -- ++(2,1.2);
			 \node[above] at (14,-5.6) {$e_2$};
			 \draw[->] (0,0) -- ++(0,3);
			 \node[left] at (0,3) {$e_1$};
			}}
\end{tikzpicture}
\hspace{1cm}
\begin{tikzpicture}[scale=.25]
	
		\draw[dotted,fill=black!5] (6, -7) --++ (1,-.6) --++ (1, .6) --++ (-1,.6) --++ (-1,-.6);
		\draw[fill=black!15!] (6,-7) --++ (0,-1.4) --++ (1,-.6) -- ++ (0,1.4) -- ++ (-1,.6);
		\draw[fill=black!15!] (6,-8.4) --++ (0,-1.4) --++ (1,-.6) -- ++ (0,1.4) -- ++ (-1,.6);
		
		\draw[dotted,fill=black!5] (7, -6.4) --++ (1,-.6) --++ (1, .6) --++ (-1,.6) --++ (-1,-.6);
		\draw[fill=black!40!] (7,-7.6) --++ (0,-1.4) --++ (1,.6) --++ (0,1.4) --++ (-1,-.6);
		\draw[fill=black!40!] (7,-9) --++ (0,-1.4) --++ (1,.6) --++ (0,1.4) --++ (-1,-.6);
			
		\draw[dotted, fill=black!20!] (8,-7) --++ (0,-1.4) --++ (1,.6) --++ (0,1.4) --++ (-1,-.6);
		\draw[dotted,fill=black!5] (8, -7) --++ (1,-.6) --++ (1, .6) --++ (-1,.6) --++ (-1,-.6);
		\draw[fill=black!15!] (8,-7) --++ (0,-1.4) --++ (1,-.6) -- ++ (0,1.4) -- ++ (-1,.6);
		\draw[fill=black!40!] (9,-7.6) --++ (0,-1.4) --++ (1,.6) --++ (0,1.4) --++ (-1,-.6);

		\foreach \x in {0,...,4}{
		\draw[fill=black!0!] (2*\x+1,.6) -- (2*\x+2,1.2) -- (2*\x+3,.6) -- (2*\x+2,0) -- (2*\x+1,.6);	
		}
		\foreach \x in {0,...,5}{
			\draw[fill=black!0!] (2*\x,0) -- (2*\x+1,.6) -- (2*\x+2,0) -- (2*\x+1,-.6) -- (2*\x,0);
			}
			
		\foreach \x in {0,...,5}{
			\draw[fill=black!0!] (2*\x+2,-14) -- (2*\x+1,-14.6) -- (2*\x,-14);
			\draw (2*\x,0) -- (2*\x,-14);
			\draw(12,0) -- (12,-14);
			\draw (2*\x+1,-.6)--(2*\x+1,-14.6);

			}
			\draw[dotted] (11,-14.6) -- ++ (-5,-3);
			\draw[dotted] (12,-6.8) -- ++ (6,-3.6);
			\draw[dotted] (12,-14)--(18,-10.4);
			 \draw[->] (0,-14) -- ++(8,-4.8);
			 \node[below] at (8,-18.8) {$e_3$};
			 \draw[->] (12,-6.8) -- ++(2,1.2);
			 \node[above] at (14,-5.6) {$e_2$};
			 \draw[->] (0,0) -- ++(0,3);
			 \node[left] at (0,3) {$e_1$};
\end{tikzpicture}
\caption{Let $n=4$ and $k=2$. On the left: The projection of the set $\mathcal{Z}^2(K;N)$ into the subspace spanned by $e_1$, $e_2$ and $e_3$ in a situation where $K=9$ and $\lfloor c_o \log{K} \rfloor = 15$ (only $10$ zig-zag steps appear, instead of $15$ because indices $t$ for which $r_t=4$ lead to steps towards a fourth dimension).
On the right: the union of the boxes corresponding to a piece of a path of boxes inside $\mathcal{Z}(K;N)$. 
In order for the event $\mathcal{O}_1$ to happen, one needs the existence of such a path crossing the region from bottom to top.}
\label{fig:wall_boxes_path}
\end{figure}
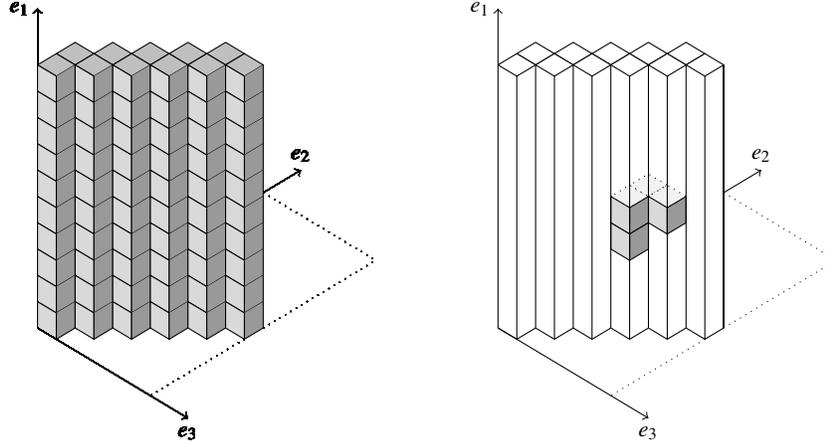
\end{center}

Recall the definition of a path of good boxes as being a directed sequence of sites, each lying at distance $N$ from the preceding one, and such that the corresponding boxes of size $N$ are good (see  \eqref{eq:path_good_boxes}).
In light of Lemma \ref{l:path_of_good}, the occurrence of $\mathcal{O}_1$ entails the existence of a path of good boxes $\{z_0,\ldots,z_T\}\subset\ZZ^n$ with $\pi_{\{1\}}(z_0)=0$ and $\pi_{\{1\}}(z_T)=NK$, and such that each good box  $B(z_i;N)$ is contained in the set $\mathcal{Z}^2(K;N)\subset \mathbb{Z}^n$ defined as
\begin{equation}
\label{eq:z2}
\mathcal{Z}^2(K;N):= \hspace{-.7cm}\bigcup_{\substack{ (t,x)\in\ZZ^2 \colon \\ 0\leq t\leq \lfloor c_o\log K \rfloor,0\leq x\leq K} } \hspace{-1cm} B(p_t+Nx e_1;N).
\end{equation}
Roughly speaking, the set $\mathcal{Z}^2(K;N)$ is the subset of $\mathbb{Z}^n$ that comprises all the sites inside boxes of size $N$ in a portion of the renormalized lattice resembling a thickening of width $N$ of a jagged rectangular region of side $c_o \log(K)$ and height $K$.
The reader might find it useful to consult Figure \ref{fig:wall_boxes_path} in order to clarify the definition of the set $\mathcal{Z}^2(K;N)$ and the definition of the event $\mathcal{O}_1$.
Note however that the picture may be a little bit misleading because for $K$ large, the jagged wall region depicted therein should look very elongated in the $e_1$ direction.

Let $\mathcal{O}_2$ be the event that every site $z\in\mathcal{Z}^2(K;N)$ is  $\omega_I$-open for all the indices $I\in\mathcal{I}(2;n)$ such that $I\cap\{1\}=\varnothing$.
Notice that, for every index $I\in \mathcal{I}(2;n)$ that does not include the coordinate $1$ each of the boxes appearing in the l.r.s.\ of \eqref{eq:z2} projects to a square in $\mathbb{Z}^2_I$ containing $N^2$ sites.
Moreover, using the cyclic nature of the $r_t$, we can conclude that as $t$ runs over the interval $0,\ldots, \lfloor c_o \log K \rfloor$, the amount of different projections into each $\mathbb{Z}^2_I$ that one needs to check in order to determine the occurrence of $\mathcal{O}_2$ does not exceed  $c' \log K$ for some positive universal constant $c' = c'(n,c_o)$ (for instance $ c'=3(n-1)^{-1} c_o$).
See Figure \ref{fig:spiral_proj} for an illustration of these projections.

\begin{center}
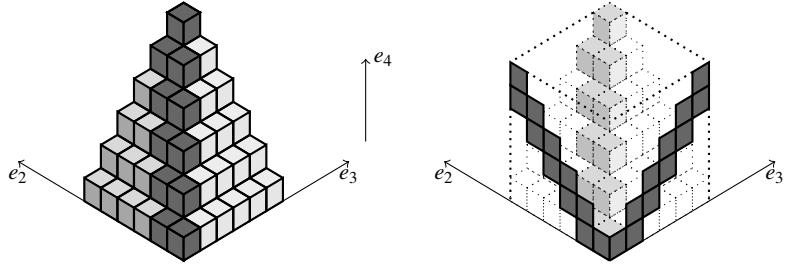
\begin{figure}[htb!]
	\centering
	\begin{tikzpicture}[scale=.22]
	
   				\foreach \x in {0,...,5}
    			{\foreach \i in {0,...,\x}
    			 {\pgfmathparse{.60*\i}
    			    \xdef\y{\pgfmathresult}
    	 		  \draw[fill=black!7!] (\x-\i+2,-2*\x-\y)--(\x-\i+1,-2*\x-\y+.60)--(\x-\i,-2*\x-\y)--(\x-\i+1,-2*\x-\y-.60)--(\x-\i+2,-2*\x-\y);
    				  \draw[fill=black!10!](\x-\i+2,-2*\x-\y)--(\x-\i+2,-2*\x-\y-1.40)--(\x-\i+1,-2*\x-\y-2)--(\x-\i+1,-2*\x-\y-.60);
    				  }}
		\foreach \x in {0,...,5}
			{\foreach \i in {0,...,\x}
			 {\pgfmathparse{.60*\i}
			    \xdef\y{\pgfmathresult}
				\draw[fill=black!15!](-\x+\i,-2*\x-\y)--(-\x+\i+1,-2*\x-\y+.60)--(-\x+\i+2,-2*\x-\y)--(-\x+\i+1,-2*\x-\y-.60)--(-\x+\i,-2*\x-\y);
				\draw[fill=black!35!] (-\x+\i,-2*\x-\y)--(-\x+\i,-2*\x-\y-1.40)--(-\x+\i+1,-2*\x-\y-2)--(-\x+\i+1,-2*\x-\y-.60)--(-\x+\i,-2*\x-\y);

				 \draw[fill=black!60!] (0,-2*\x-.60*\x)--(1,-2*\x-.60*\x+.60)--(2,-2*\x-.60*\x)--(1,-2*\x-.60*\x-.60)--(0,-2*\x-.60*\x);

				 \draw[fill=black!60!] (0,-2*\x-.60*\x)--(0,-2*\x-.60*\x-1.40)--(1,-2*\x-.60*\x-2)--(1,-2*\x-.60*\x-.60)--(0,-2*\x-.60*\x);
				 }}
				 
				 \foreach \x in {1,...,5}{
				 \draw[fill=black!60!] (-1,-2*\x-.60*\x+.60)--(0,-2*\x-.60*\x+.60+.60)--(1,-2*\x-.60*\x+.60)--(0,-2*\x-.60*\x-.60+.60)--(-1,-2*\x-.60*\x+.60);
				 \draw[fill=black!60!] (-1,-2*\x-.60*\x+.60)--(-1,-2*\x-.60*\x-1.40+.60)--(0,-2*\x-.60*\x-2+.60)--(0,-2*\x-.60*\x-.60+.60)--(-1,-2*\x-.60*\x+.60);
				 }
				 
		\foreach \x in {0,...,5}
			{\foreach \i in {0,...,\x}
			 {\pgfmathparse{.60*\i}
			    \xdef\y{\pgfmathresult}
			     \draw[thick] (-\x+\i,-2*\x-\y)--(-\x+\i+1,-2*\x-\y+.60)--(-\x+\i+2,-2*\x-\y)--(-\x+\i+1,-2*\x-\y-.60)--(-\x+\i,-2*\x-\y)--(-\x+\i,-2*\x-\y-1.40)--(-\x+\i+1,-2*\x-\y-2)--(-\x+\i+1,-2*\x-\y-.60);
			    	 \draw[fill=black!60!](\x-\x+2,-2*\x-.60*\x)--(\x-\x+2,-2*\x-.60*\x-1.40)--(\x-\x+1,-2*\x-.60*\x-2)--(\x-\x+1,-2*\x-.60*\x-.60);
				 \draw[thick] (\x-\i+2,-2*\x-\y)--(\x-\i+1,-2*\x-\y+.60)--(\x-\i,-2*\x-\y)--(\x-\i+1,-2*\x-\y-.60)--(\x-\i+2,-2*\x-\y)--(\x-\i+2,-2*\x-\y-1.40)--(\x-\i+1,-2*\x-\y-2)--(\x-\i+1,-2*\x-\y-.60);
				 }}

\draw[->] (1,-15) -- ++(-10,6);
\node[below] at (-9,-9) {$e_2$};
\draw[->] (1,-15) -- ++(10,6);
\node[below] at (11,-9) {$e_3$};
\draw[->] (12,-7.8)--(12,-2.8);
\node[right] at (12,-2.8) {$e_4$};

	\end{tikzpicture}	
	\hspace{.2cm}
	\begin{tikzpicture}[scale=.22]
	
		\foreach \x in {0,...,5}{
				 \draw[thin, dotted, fill=black!15!] (0,-2*\x-.60*\x)--(1,-2*\x-.60*\x+.60)--(2,-2*\x-.60*\x)--(1,-2*\x-.60*\x-.60)--(0,-2*\x-.60*\x);

				 \draw[thin, dotted, fill=black!25!] (0,-2*\x-.60*\x)--(0,-2*\x-.60*\x-1.40)--(1,-2*\x-.60*\x-2)--(1,-2*\x-.60*\x-.60)--(0,-2*\x-.60*\x);
				  \draw[thin, dotted, fill=black!7!] (1,-2*\x-.60*\x-.60) -- ++ (0,-1.4) -- ++ (1,.6) -- ++ (0,1.4) -- ++ (-1,-.6);
				 }
				 
				  \foreach \x in {1,...,5}{
				 \draw[dotted, thin, fill=black!15!] (-1,-2*\x-.60*\x+.60)--(0,-2*\x-.60*\x+.60+.60)--(1,-2*\x-.60*\x+.60)--(0,-2*\x-.60*\x-.60+.60)--(-1,-2*\x-.60*\x+.60);
				 \draw[dotted, thin, fill=black!25!] (-1,-2*\x-.60*\x+.60)--(-1,-2*\x-.60*\x-1.40+.60)--(0,-2*\x-.60*\x-2+.60)--(0,-2*\x-.60*\x-.60+.60)--(-1,-2*\x-.60*\x+.60);
				 }

		\foreach \x in {0,...,5}
			{\foreach \i in {0,...,\x}
			 {\pgfmathparse{.60*\i}
			    \xdef\y{\pgfmathresult}
			     \draw[thin, dotted] (-\x+\i,-2*\x-\y)--(-\x+\i+1,-2*\x-\y+.60)--(-\x+\i+2,-2*\x-\y)--(-\x+\i+1,-2*\x-\y-.60)--(-\x+\i,-2*\x-\y)--(-\x+\i,-2*\x-\y-1.40)--(-\x+\i+1,-2*\x-\y-2)--(-\x+\i+1,-2*\x-\y-.60);
				 \draw[thin, dotted] (\x-\i+2,-2*\x-\y)--(\x-\i+1,-2*\x-\y+.60)--(\x-\i,-2*\x-\y)--(\x-\i+1,-2*\x-\y-.60)--(\x-\i+2,-2*\x-\y)--(\x-\i+2,-2*\x-\y-1.40)--(\x-\i+1,-2*\x-\y-2)--(\x-\i+1,-2*\x-\y-.60);
				 }}
				 
				 \foreach \x in {0,...,5}{
				 \draw[thick, fill=black!60!] (1-\x, -15+2*\x) -- ++ (-1,.6) -- ++ (0,1.4) -- ++(1, -.6) -- ++(0,-1.4);}
				 \foreach \x in {0,...,4}{
				 \draw[thick, fill=black!60!] (-\x, -14.4+2*\x) -- ++ (-1,.6) -- ++ (0,1.4) -- ++(1, -.6) -- ++(0,-1.4);
				 }
				 
				  \foreach \x in {0,...,5}{
				 \draw[thick, fill=black!60!] (1+\x, -15+2*\x) -- ++ (1,.6) -- ++ (0,1.4) -- ++(-1, -.6) -- ++(0,-1.4);}
				 \foreach \x in {0,...,4}{
				 \draw[thick, fill=black!60!] (2+\x, -14.4+2*\x) -- ++ (1,.6) -- ++ (0,1.4) -- ++(-1, -.6) -- ++(0,-1.4);
				 }

\draw[dotted, thick] (0,0) -- ++ (-5,-3);
\draw[dotted, thick] (-5,-5.8) -- ++ (0,-5.6);
\draw[dotted, thick] (2,0) -- ++ (5,-3);
\draw[dotted, thick] (7,-5.8) -- ++ (0,-5.6);
\draw[dotted, thick] (-4,-3.6) -- ++ (5,-3);
\draw[dotted, thick] (6,-3.6) -- ++ (-5,-3);
\draw[->] (1,-15) -- ++(-10,6);
\node[below] at (-9,-9) {$e_2$};
\draw[->] (1,-15) -- ++(10,6);
\node[below] at (11,-9) {$e_3$};
	\end{tikzpicture}
	\caption{On the left: the darker region is the projections of the $\mathcal{Z}^2(K;N)$ into the subspace spanned by $e_2$, $e_3$ and $e_4$.
	The spiral contains $c_o \lfloor \log{K} \rfloor$ boxes of size $N$.
	On the right: the projection of the boxes into $\mathbb{Z}^2_{\{2,4\}}$ and $\mathbb{Z}^2_{\{3,4\}}$. 
	These projections contain no more than $c' \log{K}$ squares with $N^2$ sites (for some constant $c' =c' (c_o,n) >0$). 
	}
	\label{fig:spiral_proj}	
\end{figure}
\end{center}

Therefore, we have
\[
\mathbb{P}(\mathcal{O}_2) \geq \prod_{\substack{ I\in\mathcal{I(}2;n);\\\{1\}\cap I=\varnothing}} {p_I}^{c' \log(K) N^2 } = \exp\Big( c' {N}^2\log{K} \sum_{\substack{ I\in\mathcal{I}(2;n);\\\{1\}\cap I=\varnothing}} \log p_I\Big).
\]
Let us define
\[
\alpha_2:={c' {N}^2} \sum_{\substack{ I\in\mathcal{I}(2;n);\\\{1\}\cap I=\varnothing}} \log(\tfrac{1}{p_I})
\]
so that we have
\[
\mathbb{P}(\mathcal{O}_2) \geq \exp (-\alpha_2 \log{K}) = K^{-\alpha_2}
\]
which is to say that the probability of $\mathcal{O}_2$ is bounded below by a term that is proportional to a negative power of $K$ with exponent $\alpha_2$.

By Lemma \ref{l:path_of_good}, on the event $\mathcal{O}_1$ there exists a path $\gamma:\{0\}\cup [T]\to \mathcal{Z}^2(K;N)$ of sites that are $\omega_{I_j}-$open for all $2\leq j \leq n$ such that $\pi_{\{1\}}(\gamma(0))=0$ and $\pi_{\{1\}}(\gamma(T))=NK$.
On the event $\mathcal{O}_1\cap \mathcal{O}_2$, the sites in $\gamma$ are actually $\omega$-open.
Unfortunately, this path may not start at $o$. In order to fix this issue, let us introduce the event $\mathcal{O}_3$  that all the sites $z\in \mathbb{Z}^n$ with $\pi_{\{1\} }(z)=-1$ and such that $z+e_1\in \mathcal{Z}^2(K;N)$ are $\omega_{I_j}$-open for each $2\leq j \leq n$.
Since there are no more than $cN^{n-1} \log{K}$ such sites, $\mathbb{P}(\mathcal{O}_3)$ is also proportional to a negative power of $K$.
Moreover, on the event $\mathcal{O}_2\cap\mathcal{O}_3$ these sites are $\omega$-open.
 
Now, on the event $\mathcal{O}_1\cap \mathcal{O}_2\cap \mathcal{O}_3$, $o$ and $\gamma(0)$ are connected by a path of $\omega$-open sites $\lambda:\{0\}\cup [J]\to\ZZ^n$, with $\lambda(0)=o$, $\lambda(J) = \gamma(0)$ and  $\pi_{\{1\}}( \lambda(j)) =-1$ for every $j = 1,\ldots, J-1$. 
In other words, we have:
\begin{equation}
\label{eq:o1o2o3}
\mathcal{O}_1\cap \mathcal{O}_2\cap \mathcal{O}_3\subset[ o \leftrightarrow \partial{B}(NK)].
\end{equation}

Having constructed the events $\mathcal{O}_1$, $\mathcal{O}_2$ and $\mathcal{O}_3$ whose occurrence implies the existence of a long $\omega$-open path starting at the origin, we now construct events on which the cluster containing the origin is finite.
This is done by first requiring that the origin in $\mathbb{Z}^2_{\{2,3\}}$ is enclosed by a circuit composed of $\omega_{\{2,3\}}$-closed sites only and then requiring that there exists a set $S$ surrounding the origin in $\mathbb{Z}^{n-2}_{[n]\setminus \{2,3\}}$ whose edges are $\mathcal{P}_{\{2,3\}}(x)$-closed for every for every $x$ inside the circuit.
 
Indeed, let $\mathcal{O}_4$ be the event in which all sites in the square circuit in $\mathbb{Z}^2_{\{2,3\}}$ given by
\[
\partial \bigg(([-2, 4Nc_o\lfloor\log K\rfloor+1]\times[-2, 4Nc_o\lfloor\log K\rfloor+1]) \cap \ZZ_{\{2,3\}}^2 \bigg)
\]
are $\omega_{{\{2,3\}} }$-closed.
Since the perimeter of this circuit is less than $c N \log{K}$, for some $c>0$, the probability that $\mathcal{O}_4$ occurs is bounded below by a negative power of $K$. 
Moreover, on the event $\mathcal{O}_4$, the origin of $\mathbb{Z}^2_{\{2,3\}}$ is surrounded by a $\omega_{\{2,3\}}$-closed circuit, therefore, $\mathcal{V}_{\omega_{\{2,3\}}}(o;\mathbb{Z}^2_{\{2,3\}})$ is finite, which is to say that Condition \emph{i.} in Lemma \ref{lemma nao perc} is satisfied.
Note also that the circuit encloses the projection of $\mathcal{Z}^2(K;N)$ into $\mathbb{Z}^2_{\{2,3\}}$.

Now, similarly to Condition \emph{ii}.\ in Lemma \ref{lemma nao perc}, let $\mathcal{O}_5$ be the event in which there exists $\mathcal{S}\subset \ZZ_{[n]\setminus{\{2,3\}}}^{n-2}$ that surrounds the origin satisfying that $\inf \{\|s\|:s\in\mathcal{S} \}\geq 3NK$ and that each site $s\in\mathcal{S}$ is $\mathcal{P}_{\{2,3\}}(x)-$closed for all $x\in \big([-1, 4Nc_o\lfloor\log K\rfloor]\times[-1, 4Nc_o\lfloor\log K\rfloor]\big) \cap \ZZ_{\{2,3\}}^2 $.
Following exactly the same type of Borel-Cantelli argument as in the proof of Lemma \ref{lemma nao perc 1}, we obtain that  $\mathbb{P}(\mathcal{O}_5)=1$. 

On the event $\mathcal{O}_4\cap\mathcal{O}_5$, Conditions \emph{i.-ii}.\ in the statement of Lemma \ref{lemma nao perc} are satisfied. 
Therefore
\begin{equation}
\label{eq:o4o5}
\mathcal{O}_4\cap\mathcal{O}_5\subset [o \nleftrightarrow \infty].
\end{equation}
Combining \eqref{eq:o1o2o3} and \eqref{eq:o4o5} we get
\[
\PP_{\mathbf{p}}(o \leftrightarrow \partial{B}(NK),~ o \nleftrightarrow \infty)\geq \PP(\mathcal{O}_1\cap\mathcal{O}_2\cap\mathcal{O}_3\cap \mathcal{O}_4\cap \mathcal{O}_5),
\]
Since $\mathcal{O}_1,\mathcal{O}_2, \mathcal{O}_3$, $\mathcal{O}_4$ and $\mathcal{O}_5$ are independent events, the fact that the probabilities of $\mathcal{O}_1$ and $\mathcal{O}_5$ are uniformly  bounded below by a positive constant, and that the remaining events have probability proportional to a negative power of $K$ finishes the proof. 
\qed
\end{proof}

We now indicate a few modifications to the proof of Theorem \ref{t:poly_dec_1} that lead to Theorem \ref{theorem power law projetando em k}.

\begin{proof}[Sketch of the proof for Theorem \ref{theorem power law projetando em k}]

One major ingredient for the proof of power law decay of the truncated connectivity function in the case $k=2$ to hold in a wider range of parameters $\mathbf{p}$ is Lemma \ref{lemma dos caminhos} that has a two-dimensional appealing.
It is not clear to us how to obtain an analogous counterpart for $k\geq 3$.
Moreover, our renormalization arguments that relies on this result provides a suitable $N$ for which the process of good boxes dominates stochastically a supercritical Bernoulli percolation  requiring only that the parameters $p_{I_j}$ stay above $p_c(\ZZ^2)$ (see Lemma \ref{lemma dominacao estocastica}).

In this line of reasoning, our first modification consists in fixing $N=1$ and redefining the notion of \textit{good boxes} for boxes of type $B(y;N)$.
For $N=1$, $B(y;1)$ consists of a single point, i.e., $B(y;1)=\{y\}$. 
We say that $B(y;1)$ is good if $y$ is $\omega_I$-open for all $I\in\mathcal{I}(k;n)$ for which $I\cap \{1\} \neq \varnothing$.

With this notion of \textit{good boxes}, a path of good boxes (as appearing in the statement of Lemma \ref{l:path_of_good}) is simply a path of sites that are $\omega_I$-open for all $I\in\mathcal{I}$ that contains $1$ as an element, \textit{i.e.}, Lemma \ref{lemma dos caminhos}  holds trivially.
Furthermore, Lemma \ref{lemma da cortina chi dependencia} holds with $N=1$. 
In particular, if for all $I\in\mathcal{I}(k;n)$ for which $I\cap\{1\}=\{1\}$ we choose all parameters $p_I<1$ to be close enough to $1$, then stochastic domination as in Lemma \ref{lemma dominacao estocastica} holds with $1-\epsilon>p_c(\ZZ^2)$.

The rest of the proof follows by fixing $N=1$ and repeating the proof of Theorem \ref{t:poly_dec_1} with some minor adaptations. 
\qed
\end{proof}

\begin{acknowledgement}
The three first authors M.A., M.H.\ and B.N.B.L.\ were former PhD students of the fourth author Vladas Sidoravicius. 
An advanced draft of this paper was produced before Vladas passed away in May 2019.
M.A., M.H.\ and B.N.B.L.\ would like to thank Vladas for his great intuition, ideas and enthusiasm that played a fundamental role in our academic trajectories and in obtaining the results in this paper.
The research of M.H.\ was partially supported by CNPq grants `Projeto Universal' (406659/2016-8) and `Produtividade em Pesquisa' (307880/2017-6) and by FAPEMIG grant `Projeto Universal' (APQ-02971-17).  The research of B.N.B.L.\ was supported in part by CNPq grant 305811/2018-5 and FAPERJ (Pronex E-26/010.001269/2016).
We are also very grateful to Alberto Sarmiento for suggesting a strategy for proving Lemma \ref{lemma construcao da base da rede} and to the anonymous referee for pointing out several corrections and for providing valuable suggestions.
\end{acknowledgement}

\section*{Appendix}
\label{ap:linear_algebra}
\addcontentsline{toc}{section}{Appendix}

In this appendix we present a proof of Lemma \ref{lemma construcao da base da rede} that relies on elementary linear algebra arguments.

\begin{proof}[Proof of Lemma \ref{lemma construcao da base da rede}] Let $L:\RR^n\to\RR^{n-2}$ be the linear application: $L(x_1,\ldots,x_n)=(y_1,\ldots,y_{n-2})$ where
$$ y_j = x_j +x_{n-1}+jx_n, \mbox{ for each } 1\leq j\leq  n-2.$$
Let $v_1,v_2\in\RR^n$ be the vectors:
\begin{align*} v_1=&(-1,\ldots,-1,1,0),\\v_2=&(-1,-2,-3,\ldots,-(n-2),0,1).
\end{align*}
Let $U:\RR^2\to\RR^n$ be the linear application $U(x,y)=xv_1+yv_2$, and denote
\begin{align*}
&\Ker L:=\{v\in\RR^n:L v=0\},\\
&\Ran U = \{U u:u\in\RR^2\}.
\end{align*}
We claim that $\Ker L = \Ran U$, and that $\pi_I\circ U$ is injective for all $I\subset [n]$ with $\# I\geq 2$.

In fact, from the definition of $L$ we obtain that $v=(x_1,\ldots,x_n)\in \Ker L$ if and only if $L v=(y_1,\ldots,y_{n-2})$ satisfies $y_j=0$ for all $1\leq j \leq n-2$, and this equality holds if and only if $x_j=-( x_{n-1}+jx_n).$ In particular, $v\in\Ker L$ if and only if
$$v=x_{n-1}v_1+x_nv_2.$$
This shows that $\Ker L= \Ran U$.

To prove the second statement of our claim, we begin by observing the fact that:\\
\noindent \textit{For each $I\subset[n]$ with $\# I= 2$ the linear application $\pi_I\circ U: \RR^2\to \RR^{2}$ is injective.}\\
To see that this is true, let $I=\{i,j\}$, with $1\leq i <j \leq n$.
Then the possibilities for the matrix $\pi_I\circ U$ are:
\[
a)\left(\begin{array}{cc}
    -1 & -i \\
    -1 & -j \\
  \end{array}\right),\;
  b)\left(\begin{array}{cc}
    -1 & -i \\
     1 & 0 \\
  \end{array}\right),\;
c)\left(\begin{array}{cc}
    -1 & -i \\
     0 & 1 \\
  \end{array}\right),\;\mbox{ and }
  d)\left(\begin{array}{cc}
     1 & 0 \\
     0 & 1 \\
  \end{array}\right),
  \]
where $a)$ corresponds to $j\leq n-2$; $b)$ corresponds to $j=n-1$; $c)$ corresponds to $i\leq n-2$ and $j=n$; $d)$ corresponds to $i=n-1$. In any case we have
$|\det \pi_I\circ U|\geq 1$ which implies injectivity.

It follows now that for all $I\subset [n]$ with $\# I\geq 2$ the linear application $\pi_I \circ U:\RR^2\to\RR^{\# I}$ is injective. In fact, if for $u,v\in\RR^2$ we have $\pi_I\circ U u=\pi_I\circ U v$, then in particular, for each $J\subset I$ with $\#J=2$ we have $\pi_J\circ U u=\pi_J\circ U v$ which implies $u=v$.

Since $v_1$ and $v_2$ are linearly independent (over $\RR$), by the Gram-Schmidt process there exists $\tilde{w_1},\tilde{w_2}\in\RR^n$ orthogonal and such that their linear span is equal to the linear span of $v_1$ and $v_2$:
\begin{align*}
\tilde{w}_1&=v_1,\\
\tilde{w}_2&=v_2-\frac{\langle v_1,v_2\rangle}{\langle v_1,v_1\rangle}v_1\\
   &=v_2+\frac{(2-n)}{2}v_1.
\end{align*}
In particular, $w_1= 2\|\tilde{w}_2\|\tilde{w}_1$ and $w_2=2\|\tilde{w}_1\|\tilde{w}_2$ are orthogonal, $\|w_1\|=\|w_2\|$ and both belong to $\ZZ^n$. 
Hence the linear application $A:\ZZ^2\to\RR^n$ given by
\begin{align*}
A(x,y)=&xw_1+yw_2\\
=&U(2x\|\tilde{w}_2\|+y(2-n)\|\tilde{w}_1\|,2y\|\tilde{w}_1\|)
\end{align*}
is such that $\Ran A \subset \Ker L \cap \ZZ^n$. 
Let $H:\ZZ^2\to\ZZ^2$ be the linear application
$$H=\left(
      \begin{array}{cc}
        2\|\tilde{w}_2\| & (2-n)\|\tilde{w}_1\| \\
        0 & 2\|\tilde{w}_1\| \\
      \end{array}
    \right).$$
Then $H$ is injective and $A=U\circ H$. Hence for each $I\subset [n]$ with $\# I\geq2$, the linear application $\pi_I \circ A =\pi_I \circ U \circ H$ is injective, since $H$ and $\pi_I \circ U$ are injective.
Defining
\[
c=\min_{I\subset [n] }\inf_{\substack{x\in\RR^2\\\|x\|=1} } \|\pi_I\circ A x\| >0
\]
completes the proof.  
\qed
\end{proof}

\input{referenc}

\end{document}

%% file: referenc.tex
%
%
%

%% file: revised_2020_08_07.bbl
\begin{thebibliography}{99.}

\bibitem{Aizenman87}
Aizenman, M., Barsky, D.J.:, Sharpness of the phase transition in percolation models, Comm. Math. Phys. \textbf{108}, 489--526 (1987)

\bibitem{broadbent57}
Broadbent, R.S., Hammersley, J.M.: Percolation processes: I. Crystals and mazes. Mathematical Proceedings of the Cambridge Philosophical Society \textbf{53}, 629--641 (1957)

\bibitem{Duminil-Copin16}
 Duminil-Copin, H., Tassion, V.: A new proof of the sharpness of the phase transition for Bernoulli percolation and the Ising model. Comm. Math. Phys. \textbf{343}, 725--745 (2016)

\bibitem{gacs_2000}
G\'acs, P.: The clairvoyant demon has a hard task. Combinatorics, Probability and Computing \textbf{9}, 421--424 (2000)

\bibitem{Grassberger17}
Grassberger, P.: Universality and asymptotic scaling in drilling percolation. Phys. Rev. E. \textbf{95}, p.~010103 (2017)

\bibitem{Grassberger17_2}
Grassberger, P., Hil{\'a}rio, M.R., Sidoravicius V.: Percolation in media with columnar disorder. J. Stat. Phys. \textbf{168}, 731--745 (2017)

\bibitem{HilSid}
Hil\'ario M., Sidoravicius V.: Bernoulli line percolation. Stoch. Proc. Appl. \textbf{129}, 5037--5072 (2019)

\bibitem{Kantor86}
Kantor, Y.: Three-dimensional percolation with removed lines of sites, Phys. Rev. B. \textbf{33}, 3522--3525 (1986)

\bibitem{Ligget}
Liggett T.M., Schonmann, R.H., Stacey, A.M.: Domination by product measures, Ann. Probab. \textbf{25}, 71--95 (1997)

\bibitem{Menshikov86}
Menshikov, M.V.: Coincidence of critical-points in percolation problems. Soviet Math. Dokl. \textbf{25}, 856--859 (1986)

\bibitem{G.Pete}
Pete G.: Corner percolation on {$\Bbb Z^2$} and the square root of 17. Ann. Probab. \textbf{36}, 1711--1747 (2008)

\bibitem{Schrenk16}
Schrenk, K.J., Hil\'ario, M.R., Sidoravicius,V., Ara\'ujo N.A.M., Herrmann H.J., Thielmann, M., Teixeira, A.: Critical fragmentation properties of random drilling: How many holes need to be drilled to collapse a wooden cube? Phys. Rev. Lett. \textbf{116}, p.~055701 (2016)

\bibitem{winkler_2000}
Winkler, P.: Dependent percolation and colliding random walks. Random Struct. Algor. \textbf{16}, 58--84 (2000)

\end{thebibliography}
